%% file: currentversion.tex
\documentclass[10pt]{amsart}

\usepackage{amsmath}
\usepackage{amsfonts}
\usepackage{amssymb}
\usepackage{latexsym}
\usepackage{graphicx}
\usepackage{enumerate}
\usepackage{multirow}
\usepackage{subfigure}

\usepackage{url}

\newtheorem{theorem}{Theorem}
\newtheorem{proposition}{Proposition}
\newtheorem{lemma}[theorem]{Lemma}
\newtheorem{cor}[theorem]{Corollary}
\newtheorem{remark}{Remark}

\newcommand{\ep}{\varepsilon}

\newcommand{\be}{\begin{equation}}
\newcommand{\ee}{\end{equation}}
\newcommand{\bes}{\begin{equation*}}
\newcommand{\ees}{\end{equation*}}
\newcommand{\R}{{\bf{R}}}

\newcommand{\ds}{\displaystyle}
\newcommand{\LV}{\left\vert}
\newcommand{\RV}{\right\vert}

\newcommand{\hA}{\widehat{A}}
\newcommand{\htau}{\widehat{\tau}}

\include{gsmacros}
\begin{document}

\title{Ill-posedness of  degenerate dispersive equations}
\author{David M. Ambrose}
\address{Department of Mathematics \\ Drexel University}
\email{ambrose@math.drexel.edu}
\author{Gideon Simpson}
\address{Department of Mathematics \\ University of Toronto}
\email{simpson@math.toronto.edu}
\author{J. Douglas Wright}
\address{Department of Mathematics \\ Drexel University}
\email{jdoug@math.drexel.edu}
\author{Dennis G. Yang}
\address{Department of Mathematics \\ Drexel University}
\email{gyang@math.drexel.edu}

\begin{abstract}
In this article we provide numerical and analytical evidence that some degenerate dispersive partial differential equations 
are ill-posed.  Specifically we study the $K(2,2)$ equation $u_t = (u^2)_{xxx} + (u^2)_{x}$ and the ``degenerate Airy" equation $u_t = 2 u u_{xxx}$.  For $K(2,2)$ our results are computational in nature: we conduct a series of numerical simulations  which demonstrate that data which is very small in $H^2$ can be of unit size at a fixed time which is independent of the data's size. For the degenerate Airy
equation, our results are fully rigorous: we prove the existence of a compactly supported self-similar solution which, when combined with certain scaling invariances, implies ill-posedness (also in $H^2$).
\end{abstract}

\maketitle

\section{Introduction}
Dispersion plays a pivotal role in the analysis of the governing equations
for a large number of physical phenomena, ranging from the evolution of surface water
waves to the formation of Bose-Einstein condensates.  A common theme in such 
problems is that {\it uniform} dispersive effects  frequently  control nonlinear terms which may
otherwise lead to ill-posedness.  Consider for instance the following class of
quasilinear equations of KdV type studied in \cite{Craig-Kappeler-Strauss}:
\be
\label{CKS}
u_t = g_3 u_{xxx} + g_2 u_{xx} + g_1 u_x + g_0 u+h
\ee
where $g_j=g_j(\partial_x^{j} u,...,u)$ and $h=h(x,t)$.  
There 
it is shown that such equations
are well-posed provided $|g_3| \ge c > 0$ and $g_2 \ge 0$.
The first of these conditions guarantees that 
 dispersive
effects from $g_3 u_{xxx}$ do not vanish while
the second prohibits
the term $g_2 u_{xx}$ from acting like a backwards heat operator.
Similarly, the results on existence of solutions to the quasilinear
 dispersive equations 
found in \cite{Kenig} (about Schr\"odinger equations), \cite{Ambrose-Masmoudi} (water waves)
and \cite{Simpson} (magma dynamics) 
share the feature that the dispersive effects are nondegenerate, to name just a few examples.

On the other hand, there are a number of physical problems in which
the mechanism which causes the dispersive effects is not only nonlinear
but degenerate.  Some examples can be found in the study of
sedimentation \cite{Rubinstein}, shallow water \cite{Camassa-Holm}, granular media 
\cite{Nesterenko, Daraio1, Daraio2}, numerical analysis \cite{Goodman-Lax} and elastic rods and sheets \cite{Dai-Huo,Ketcheson}.
While there are a large number of articles in which degenerate
dispersive equations (henceforth referred to as DDE) are 
derived \cite{Nesterenko, Daraio2, Daraio1, Camassa-Holm, Ketcheson, Dai-Huo, Rubinstein}, 
special solutions computed \cite{Rosenau-Hyman, Rosenau-Kashdan, Daraio1, Daraio2, Nesterenko} or
numerical computation of solutions performed \cite{Rus1, Rus2, Rosenau-Hyman, Defrutos, Levy-Shu-Yan, Chertock-Levy}, the existence
theory for the initial value problem of these equations is
largely undeveloped.\footnote{The notable exception to this is the well-studied 
Camassa-Holm equation $u_t-u_{xxt} = - 3u u_x + 2 u_x u_{xx} + u u_{xxx}$
which when rewritten as a nonlocal evolution equation could
be said to be degenerate and dispersive.  This equation is integrable and
consequently many results for it do not obviously generalize.}  The articles \cite{Defrutos, Ambrose-Wright1, Galaktionov1}
prove some {\it a priori} estimates for DDE.
\cite{Ambrose-Wright1} and \cite{Ambrose-Wright2} prove existence (but not uniqueness) of solutions to a very special
class of DDE. \cite{Galaktionov2} contains
a semi-rigorous discussion of existence issues for these sorts of equations,
including defining ``$\delta$-entropy solutions" by borrowing ideas from first
order conservation laws.
None of these articles fully settles the well-posedness issues
for degenerate dispersion, and
thus our interest is in studying the Cauchy problem for some simple
and common DDE.  

In particular we study
the following two equations:
\be \label{K22}
u_t = 2 u u_{xxx} + 6 u_x u_{xx} + 2 u u_x
\ee
and
\be \label{DDE}
u_t = 2 u u_{xxx}
\ee
for $x \in \R$.
The first of these is the $K(2,2)$ equation of Rosenau \& Hyman \cite{Rosenau-Hyman}.  
Their goal was to understand the role of degenerate
nonlinear dispersion in the formation of patterns and this equation
(together with other members of the $K(m,n)$ family $u_t = (u^m)_x + (u^n)_{xxx}$)
has come to serve as the paradigm for degenerate dispersive systems.
The most celebrated feature of this equation is the existence
of compactly supported
traveling waves, a.k.a.~``compactons."
These are given by
$$
{u(x,t)=Q_\lambda(x+\lambda t):= {4 \lambda \over 3} \cos^2\left( {x+\lambda t \over 4} \right) \chi_I(x+\lambda t)}
$$
where $\chi_I$ is the indicator function for the interval
$I=[-2 \pi, 2 \pi]$ and $\lambda \ne 0$. More recently, $K(2,2)$
has been used to model pulses in both elastic
and granular media \cite{Ketcheson, Kevrekidis}.
Equation \eqref{DDE}, a ``degenerate Airy equation", appeared in \cite{Galaktionov1} 
and is amongst the simplest equations which could be said to feature degenerate dispersion.
We study it here due to its similarity to the $K(2,2)$ equation and because  it is of the form of \eqref{CKS} but does not meet the uniform dispersion
hypothesis needed in \cite{Craig-Kappeler-Strauss}.

Surprisingly, our conclusions are that both equations are ill-posed for data in $H^2$.
In particular, their solutions do not depend smoothly
on their initial conditions.  Our results for \eqref{K22} are computational in nature:  we conduct a 
series of numerical simulations of \eqref{K22} which demonstrate that data which is very small
in $H^2$
can be of unit size at a fixed time which is independent of the data's size.  For \eqref{DDE} our results are fully rigorous:  we prove 
the existence of a compactly supported self-similar solution which, when combined with 
certain scaling invariances, implies ill-posedness.  Note that our focus on $H^2$ initial data
is due to the fact the compacton solutions of $K(2,2)$ have a discontinuity in their second derivative
at the edge of their support \cite{Li-Olver-Rosenau}.  Consequently, they reside in $H^2$ but not $H^3$, and so it is natural
to work in this space.  

Our suspicion that \eqref{K22} is ill-posed
has its origin in the various numerical simulations of collisions between compactons in $K(m,n)$ equations.
A wide variety of numerical methods has been used to compute solutions:
pseudo-spectral \cite{Rosenau-Hyman}, particle \cite{Chertock-Levy}, Pad\'e \cite{Rus1, Rus2}, discontinuous Galerkin \cite {Levy-Shu-Yan},
and finite difference \cite{Defrutos}.  Each of these simulations shows that the computed
collisions between compactons are nearly elastic: the compactons emerge from the collision
nearly identical to their original state but phase shifted.  As is common in 
collisions between solitary waves in ``nearly" integrable PDE, the simulations 
show 
that in addition to the outgoing waves there is a small amplitude disturbance
behind the waves.  In nondegenerate problems, similar disturbances are sometimes called ``dispersive ripples"
and are viewed as being  linear effects---the tail (approximately) solves the linearized equation \cite{Wayne-Wright, Wright}.
However, since the linearization of \eqref{K22} about zero is trivial
this heuristic cannot apply.  The disturbance is described variously as: a compacton/anti-compacton pair in \cite{Rosenau-Hyman};
a numerical artifact in \cite{Rus1, Defrutos}; and as a shock in \cite{Rus2}.  No such disturbance is present at all in \cite{Chertock-Levy}.  
Lastly, the authors of \cite{Levy-Shu-Yan} describe a ``high-frequency oscillation" but go on to 
state ``we do not know what is the source of these oscillations..."
 In each case, the qualitative features
of the disturbance are quite different.  Our feeling is that these discrepancies are not related to differences
in the numerical methods or the implementation thereof but rather are due to an underlying
issue with the equation itself.

Moreover,  inspection of  the right hand side of \eqref{K22} provides
evidence that the equation may be ill-posed.
Observe that when $u_x < 0$, the term $6 u_x u_{xx}$
is formally a backwards heat operator.  
The key question, initially raised in \cite{Rosenau-Hyman}, is this: {\it
can dispersive effects due to $2u u_{xxx}$ 
counteract the instability caused by the backwards heat operator $6u_x u_{xx}$? } 
Thus we should seek to understand the dispersive effects due to $2 u u_{xxx}$---hence our interest in \eqref{DDE}.  
Our results show that the answer to the key question is ``no".  
As we stated above, our results show that $2 u u_{xxx}$ is, in its own right, just as problematic as the backwards heat term.
Differentiation of \eqref{DDE} with respect to $x$ yields for $w := u_x$
$$
w_t = 2 u w_{xxx} + 2 u_x w_{xx}.
$$
Note that the second term on the RHS above is also a backwards heat operator when $u_x < 0$. The
backwards heat term evident in \eqref{K22} is, in this way, hidden in \eqref{DDE}.
Thus we expect blow up of $u_x$ when $u_x$ is negative.  Moreover,
this blow up should be particularly catastrophic
when the function crosses the $x$-axis with negative slope.  

This motivates our approach:  the self-similar solution constructed has such a transverse crossing and the 
resulting ill-posedness is due to blow up in the first and second derivatives of the solution.
While our proof that \eqref{DDE}
is ill-posed does not naturally carry over to \eqref{K22}, the numerics are highly suggestive.  Most tellingly,
since the numerics are performed using strictly positive initial data, they suggest the ill-posedness can manifest
even in the absence of such a transverse crossing.

\begin{remark}
In \cite{Craig-Goodman}, Craig \& Goodman demonstrate that the differential equation $u_t = -x u_{xxx}$ is ill-posed, whereas
$u_t = x u_{xxx}$ is not only well-posed but the solution map is essentially infinitely smoothing.  The proofs
of these facts follow from explicit formula
for solutions of the initial value problems which can be constructed for each equation by mean of the Fourier transform
and the method of characteristics.  
\end{remark}

The remainder of this paper is organized as follows.   In Section \ref{SS} we show that the existence of a self-similar
solution to \eqref{DDE} gives, as a consequence, the ill-posedness of the equation. In Section \ref{Construct}
we construct the  self-similar solution using techniques from spatial dynamics.
In Section \ref{SC} we discuss our numerical strategy for demonstrating
ill-posedness of \eqref{K22} and present the results of our simulations.

\vspace{.1in}
\noindent {\bf Acknowledgements:}  We would like to thank Professor Philip Rosenau for many interesting and helpful conversations regarding the degenerate dispersive equations studied in this article.  Additionally, the work was supported by NSF Grants DMS 0926378,
DMS 1008387,
and DMS 1016267 (D.M.A.) and DMS 0908299 and DMS 0807738 (J.D.W.).  
G.~S. was partially funded by NSERC.
 Finally, we are especially grateful the Louis and Bessie Stein Family Foundation who generously supported this work.

%For some compacton solutions (particularly for the $K(2,2)$ equation), the derivative of the solution tends to zero at the edges of the support and so the contribution
%from the backwards heat operator also degenerates.  It is possible that in this case there is still balance between the forces.  
%On the other hand,
%notice that if a solution of $u(x,t)$ of \eqref{K02} were to cross the $x$-axis
%with a positive slope,  at the point of crossing the dispersive term vanishes
%while the backwards heat term does not.  Thus it is natural to ask if such traversals
%``blow up" at the crossing.  We now construct self-similar solutions which do exactly this.
%
%
%
%Of course, its not clear what the dispersive effects of $2 u u_{xxx}$ are, hence
%our analysis of \cite{DDE}.
%

\section{Self-similar solutions imply ill-posedness for \eqref{DDE}}
\label{SS}

%\subsection{Set-up and the proof of ill-posedness}
We will prove there is a self-similar solution for \eqref{DDE} of the form
$$
u(x,t) = A\left(\ds \frac{x}{ (1-t)^{1/3}}\right)
$$
where  $A$ gives the profile.  Though other choices for the scaling
are possible, we choose this particular set so that our proposed solution
exhibits blow up (as $t \to 1^-$) in the first and second derivative while the $L^\infty$ norm
remains bounded.
Inserting this into
\eqref{DDE} yields:
\be
\label{ODE}
 A''' =  {\tau A' \over 6A}.
\ee
Here $\tau = x (1-t)^{-1/3}$ is the self-similar coordinate upon which $A$ depends
and 
we have $(\square)' = d/d\tau(\square)$.  We specify $A(0) = 0$ and $A'(0) = -1$
so that the self-similar solution crosses the $x$-axis transversely with negative slope,
the ``bad" situation identified in the Introduction.

We prove the following proposition concerning solutions of \eqref{ODE}:
\begin{proposition}\label{ODE Prop}
There exists $\mu_* > 0$, $\tau_*>0$, $C_*>1$ and a function $A_*(\tau) \in C^3 \left( [0,\tau_*) \right)$
with the following 
properties:
\begin{enumerate}[{\rm (i)}]
\item $\ds A_*(\tau)$ solves \eqref{ODE} for all $\tau \in [0,\tau_*)$.
\item $A_*(0)=0$, $A_*'(0) = -1$ and $A_*''(0) = \mu_*$.
\item $A_*(\tau)$ is negative in $(0,\tau_*)$ and has a unique minimum in the interior of this interval.
\item $\ds \lim_{\tau \to \tau_*^-} A_*(\tau) = \lim_{\tau \to \tau_*^-} A'_*(\tau) = 0$.
\item $C_*^{-1} \left \vert \log\left |\tau-\tau_*\right |\right \vert \le \vert A''_*(\tau) \vert \le  C_*\left \vert \log\left |\tau-\tau_*\right |\right \vert $ for all $\tau \in (0,\tau_*)$ sufficiently close to $\tau_*$.
\end{enumerate}
\end{proposition}

The proof of this proposition is technical and difficult and we postpone it until the next section.
First we utilize it to prove our main result for \eqref{DDE}:
\begin{theorem}\label{DDE ip}
The Cauchy problem for \eqref{DDE}, posed in $H^2$, does not depend smoothly on the initial data.
\end{theorem}

\begin{proof}(Theorem \ref{DDE ip})
The proof of the theorem is classical:  we construct a sequence
of solutions which are arbitrarily small in $H^2$ at $t=0$
but which are arbitrarily large when $t= 1$. 
Let
$$
A(\tau) := \begin{cases}
A_*(\tau)& \textrm{ if } \tau\in[0,\tau_*)  \\
-A_*(-\tau)&\textrm{ if } \tau\in(-\tau_*,0] \textrm{ and } \\
0 &\textrm{ otherwise}.
\end{cases}
$$
This is the profile function for our self-similar solution.  See Figure~\ref{ssbup}.
\begin{figure}
\includegraphics[width=0.8\textwidth]{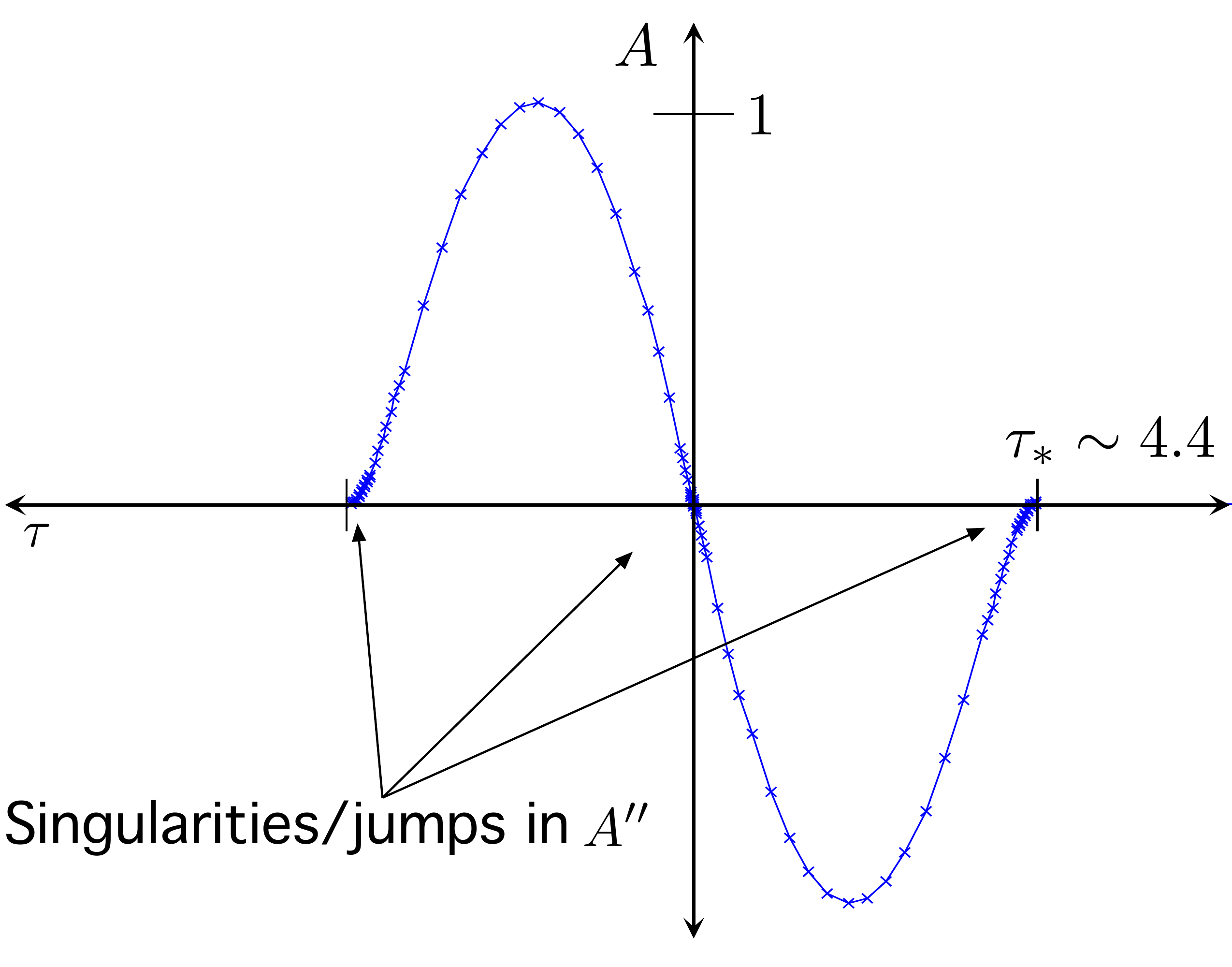}%
\caption{Numerically computed profile $A(\tau)$.  Note $\mu = \mu_\star \sim 0.354875$.\label{ssbup}}
\end{figure}

Note that it is compactly supported, not unlike the famous 
Barenblatt solution of the porous medium equation \cite{Barenblatt}.  Also observe that $A''$ is discontinuous at $\tau=0$ and 
diverges logarithmically at $\pm \tau_*$.  Nevertheless, these singularities are sufficiently mild so that 
$A \in H^2$.  Also observe that $\tau A'$ is continuous for all $\tau$ and consequently so is $A A'''$.  Thus
the function
$$
u_{1,1}(x,t):=A\left(\ds \frac{x}{ (1-t)^{1/3}}\right)
$$
is in fact a classical solution of \eqref{DDE}.  That is to say, while $\partial_x^3 u_{1,1} \notin L^2$, $u_{1,1} \partial_x^3 u_{1,1}$ is. 
Compare this with the compacton solution of $K(2,2)$: it is not in $H^3$, though its square is.

For any $T \in \R$ and $\lambda > 0$,  \eqref{DDE} is invariant under the following transformation
\be
\label{DDEscaling}
u(x,t) \longrightarrow \lambda u\left( {x \over \lambda^{1/3}},t-T\right).
\ee
Thus we have a two-parameter family of self-similar solutions of the form:
$$
u_{\lambda,T}(x,t) := \lambda A\left( {x \over \lambda^{1/3} (T-t)^{1/3} }\right).
$$
An elementary computation shows that 
\begin{multline}\label{sobscaling}
\|u_{\lambda,T}(\cdot,t) \|_{H^2} = 
\lambda ^{7/6}(T-t)^{1/6}\| A\|_{L^2} \\
+\lambda ^{5/6}(T-t)^{-1/6}\| A'\|_{L^2}
+\lambda ^{1/2}(T-t)^{-1/2}\| A''\|_{L^2}.
\end{multline}
Fixing $0<\ep \ll 1$, setting $\lambda(\ep) = \ep^{2}$ and $T(\ep)=1+\ep^{16}$
results in
$$
\|u_{\lambda(\ep),T(\ep)}(\cdot,0) \|_{H^2} \le C \ep 
\quad
\textrm{and}
\quad
\|u_{\lambda(\ep),T(\ep)}(\cdot,1) \|_{H^2} \ge {C \over \ep}.
$$
(Here $C>0$ is a nonessential constant.)  This pair of inequalities
demonstrates that we have constructed a sequence of initial data for
\eqref{DDE} which is arbitrarily small but whose solution at time $t=1$
is arbitrarily large.  Thus the equation is ill-posed in $H^2$.

\end{proof}

\begin{remark}
In \cite{Galaktionov1}, the authors investigate self-similar solutions of 
 the DDE $u_t = (u u_{x})_{xx}$.  In particular they numerically compute
 profiles for shock and rarefaction solutions which are, in many respects, quite similar
 to the ones we find here.  Namely, the shock solutions cross the $x$-axis
 with negative slope and form a singularity in $u_x$ at the time of the shock formation.
The
solutions they find are not in $L^2(\R)$---either they do not converge to zero 
as $x \to \infty$ or their decay is not rapid enough to be integrable.  Our focus here is on
well-posedness issues in $H^2$ and consequently our strategies and methods
will differ from those of \cite{Galaktionov1}.
\end{remark}

\section{Construction of the self-similar solution}\label{Construct}
This section is devoted to proving Proposition \ref{ODE Prop}.
We rewrite \eqref{ODE} as the initial value problem (IVP):
\be
\label{IVP}
6AA''' = \tau A' \quad {\textrm{with} } \quad A(0) = 0,\ A'(0)=-1,\ A''(0)=\mu.
\ee
Our strategy will be as follows: first we will show that the IVP \eqref{IVP} has a unique solution, which depends on $\mu$ continuously. Next we will show that for $\mu$ sufficiently large the solution of \eqref{IVP} goes to zero in finite time. Then we will show that for $\mu = 0$ there is a smallest time $0<\tau_2 < \infty$ such that the solution of \eqref{IVP} has its first local maximum at $\tau_2$ and remains strictly negative for all $\tau \in (0, \tau_2]$. See Figure~\ref{varyingmy}. 
\begin{figure}
\includegraphics[width=0.9\textwidth]{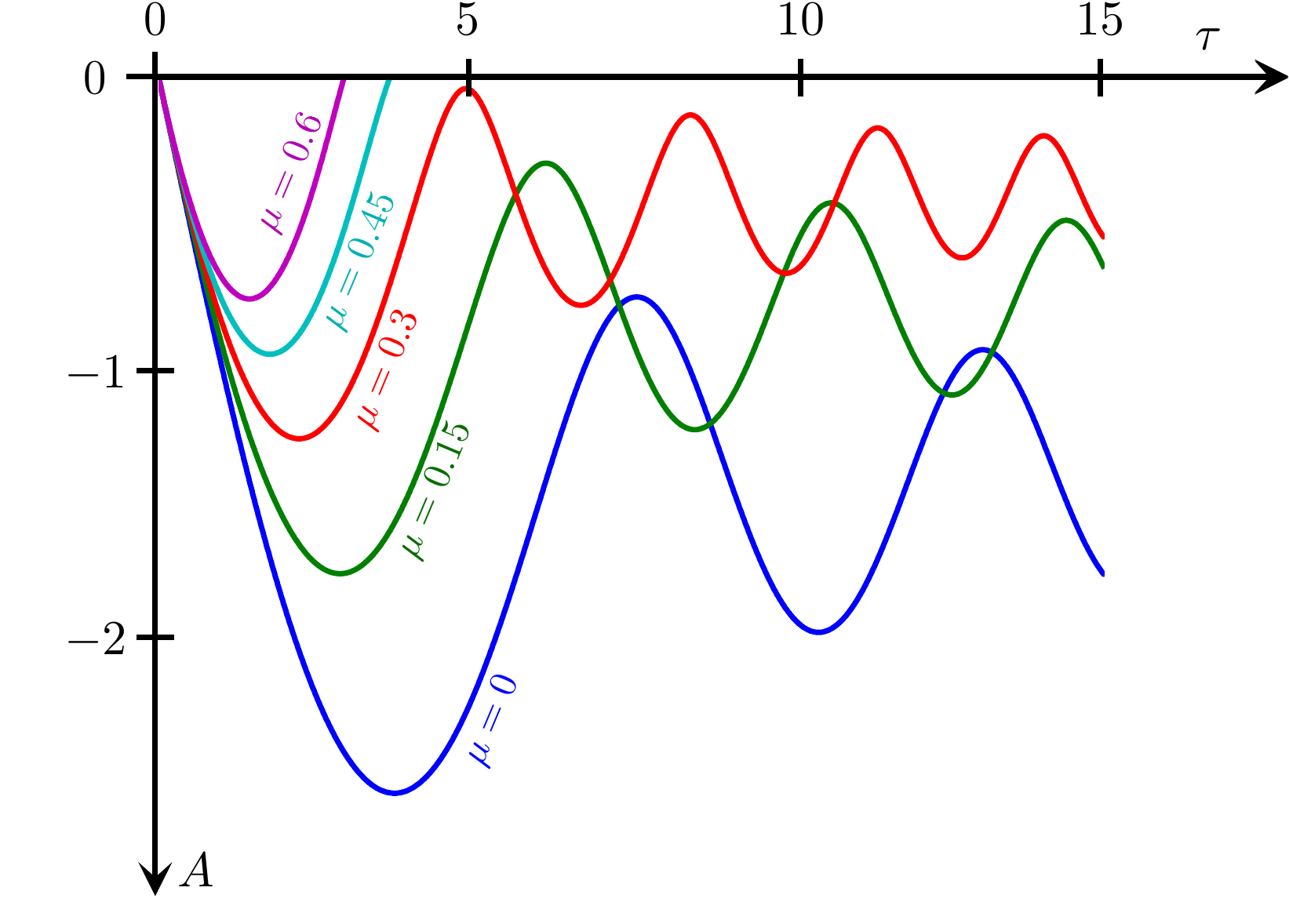}%
\caption{Numerically computed solutions of \eqref{IVP} for increasing values of  $\mu$.\label{varyingmy}}
\end{figure}
Then an open/closed argument will give the special solution $A_*$ for a choice $\mu = \mu_*>0$.

The differential equation \eqref{ODE} is nonlinear, nonautonomous and singular at $A=0$. As such, carrying out the 
details of our strategy is quite challenging.
The key observation for proving Proposition \ref{ODE Prop} is that \eqref{ODE} is, in a loose sense,
close to being hamiltonian.  Specifically, one can rewrite the right hand side of \eqref{ODE} 
as $\ds \frac{\tau}{6} \frac{d}{d\tau}\ln \LV A(\tau) \RV $.  If we then consider the related equation
\be
\label{stoopid model}
B''' = \frac{d}{d\tau}\ln \LV B(\tau) \RV 
\ee
obtained from \eqref{ODE} by replacing the nonautonomous $\tau/6$ with the constant $1$, one sees that this equation is in fact
hamiltonian.  Specifically one can integrate both sides to find:
$$
B''(\tau) =  \frac{d}{dB} \left( \left(B''(0)-1\right) B + B  \ln \LV B \RV      \right)
$$
The conserved energy is $$\ds {1 \over 2} (B')^2 -  \left( \left(B''(0)-1\right) B + B  \ln \LV B \RV      \right) $$
and most of the dynamics of \eqref{stoopid model} can be determined by analyzing this quantity.
In the following proof of Proposition \ref{ODE Prop}, we will frequently make use of quantities similar to this energy which, though not conserved for solutions of \eqref{ODE}, do provide quite a bit of dynamical information about solutions.

\subsection{Well-posedness of \eqref{IVP}}
Now we prove  \eqref{IVP} has unique solutions and that such solutions depend smoothly upon $\mu$.

\begin{lemma}\label{EUC}
For any $\mu_0 \in \R$, there exist $\tau_0 > 0$ and $\delta > 0$ such that for each $\mu \in (\mu_0 - \delta, \mu_0 + \delta)$ the IVP \eqref{IVP} has a unique solution $A : [0, \tau_0] \rightarrow \R$, which satisfies $A(\tau) < 0$ for all $\tau \in (0, \tau_0]$. In addition, there exists $C > 0$ such that for any $\mu_1, \mu_2 \in (\mu_0 - \delta, \mu_0 + \delta)$ the solutions $A_1(\tau)$ and $A_2(\tau)$ to the IVP \eqref{IVP} with $A''_1(0) = \mu_1$ and $A''_2(0) = \mu_2$, respectively, satisfy
$$
\left| A_2(\tau) - A_1(\tau) \right| + \left| A'_2(\tau) - A'_1(\tau) \right| + \left| A''_2(\tau) - A''_1(\tau) \right| \le C |\mu_2 - \mu_1|
$$
for all $\tau \in [0, \tau_0]$.
\end{lemma}

\begin{proof}
Letting $F(A,A',A'',\tau) := \ds \frac{\tau A'}{6A}$,  \eqref{ODE}
can be written as $A'''=F(A,A',A'',\tau)$. If $F(A,A',A'',\tau)$ were continuous in $\tau$ and Lipschitz in $(A,A',A'')$ throughout a neighborhood of $(A,A',A'',\tau) = (0,-1,\mu,0)$, then Lemma \ref{EUC} would be a trivial consequence of the standard theory on existence and uniqueness of solutions and their continuous dependence on initial conditions. Unfortunately, in our case $F(A,A',A'',\tau)$ is not even defined if $A = 0$. However, notice that if there exists a solution $A(\tau)$ to the IVP \eqref{IVP} then L'Hopital's rule implies:
$$
\lim_{\tau \rightarrow 0^+} \ds \frac{\tau A'(\tau)}{6A(\tau)} = \ds \frac{1}{6}.
$$
Therefore our strategy will be to remove the singularity at $(A, \tau) = (0, 0)$ by modifying $F$ in such a way that
the resulting modification is continuous in a neighborhood of the initial data and agrees with $F$
for solutions of $\eqref{IVP}$.

Fix $\mu_0 \in \R$ and $k>0$, and let $$p(\tau) := -\tau + \ds \frac{1}{2} (\mu_0 - k) \tau^2 
\quad \textrm{and} \quad q(\tau) := -\tau + \ds \frac{1}{2} (\mu_0 + k) \tau^2.$$ 
Given the initial conditions in \eqref{IVP}, $p(\tau)$ and $q(\tau)$ formally
give crude lower and upper bounds on $A(\tau)$ for $\tau$ close to $0$.
Now construct the following function:
$$
\tilde{F}(A,A',A'',\tau) := 
\begin{cases}
\ds \frac{\tau A'}{6p(\tau)} & \text{if $\tau > 0$ and $A \le p(\tau)$},\\
\ds \frac{\tau A'}{6A} & \text{if $\tau > 0$ and $p(\tau) < A < q(\tau)$},\\
\ds \frac{\tau A'}{6q(\tau)} & \text{if $\tau > 0$ and $q(\tau) \le A$},\\
-\ds \frac{A'}{6} & \text{if $\tau \le 0$}.\\
\end{cases}
$$
 $\tilde{F}$ is well-defined and continuous throughout a small neighborhood of $(A,A',A'',\tau) = (0,-1,\mu,0)$ for any $\mu \in \R$. Then we have the existence of a solution $A(\tau)$ to the IVP: 
 $$A''' = \tilde{F}(A,A',A'',\tau)\quad {\textrm{with} } \quad A(0) = 0,\ A'(0)=-1,\ A''(0)=\mu.$$ Furthermore, it is easy to verify that there exist $\tau_0 > 0$ and $0<\delta <k$ such that for each $\mu \in (\mu_0 - \delta, \mu_0 + \delta)$ the solution $A(\tau)$ is defined for all $\tau \in [0, \tau_0]$ and satisfies $p(\tau) < A(\tau) < q(\tau) < 0$ for all $\tau \in (0, \tau_0]$. It follows that the restriction of $A(\tau)$ to $[0, \tau_0]$ is in fact a solution to the IVP \eqref{IVP}.

Note that along any solution of \eqref{IVP}, $\ds \frac{\partial}{\partial A} \left( \ds \frac{\tau A'}{6A} \right) \rightarrow \infty$ as $\tau \rightarrow 0^+$. Consequently, it is not possible to modify $F$ to a function that is Lipschitz in $A$ in a neighborhood of $(A,A',A'',\tau) = (0,-1,\mu,0)$. Therefore, we have to establish the uniqueness of the solution by other means, namely the nearly hamiltonian structure of \eqref{ODE}.

For any $\mu_1, \mu_2 \in (\mu_0 - \delta, \mu_0 + \delta)$, let $A_1(\tau)$ and $A_2(\tau)$ be solutions to the IVP \eqref{IVP} with $A''_1(0) = \mu_1$ and $A''_2(0) = \mu_2$, respectively. Recall that our choice of $\delta$ and $\tau_0$ ensures $p(\tau) < A_{1,2}(\tau) < q(\tau) < 0$ for all $\tau \in (0, \tau_0]$. Thus we can take a sufficiently small $T \in (0, \min\{1, \tau_0\}]$, independent of $\mu_1$ and $\mu_2$, such that $A_{1,2}(\tau) < -\ds \frac{1}{2} \tau < 0$ for all $\tau \in (0, T]$ and
$$
M := \max_{\tau \in (0, T]} \left| \ds \frac{A_2(\tau)}{A_1(\tau)} - 1 \right| < \ds \frac{1}{2}.
$$
Then for any $\tau \in (0, T]$, 
\be \label{A2/A1}
\left| \ln{\left( \ds \frac{A_2(\tau)}{A_1(\tau)} \right)} \right| \le \ds \frac{M}{1 - M} \le 2M,
\ee
where the equalities hold if $M=0$.

Since both $A_1$ and $A_2$ are solutions to the ODE, we have 
$$
A'''_2 - A'''_1 = \ds \frac{\tau}{6} \left( \ds \frac{A'_2}{A_2} - \ds \frac{A'_1}{A_1} \right) =  \ds \frac{\tau}{6} \left( \ds \frac{A_1}{A_2} \right) \left( \ds \frac{A_2}{A_1} \right)'.
$$
Integrating the above equation gives
$$
A''_2(\tau) - A''_1(\tau) = \ds \frac{\tau}{6} \ln{\left( \ds \frac{A_2(\tau)}{A_1(\tau)} \right)} - \ds \frac{1}{6} \int_0^{\tau} \!\ln{\left( \ds \frac{A_2(s)}{A_1(s)} \right)}\, ds + \delta\mu,
$$ 
where $\delta\mu := \mu_2 - \mu_1$. Then by incorporating \eqref{A2/A1}, we obtain that for all $\tau \in (0, T]$,
\be\label{A2''-A1''}
\left| A''_2(\tau) - A''_1(\tau) \right| \le \ds \frac{2}{3} M \tau + |\delta\mu|.
\ee
Together with the initial conditions $A_1(0) = A_2(0) = 0$ and $A'_1(0) = A'_2(0) = -1$, \eqref{A2''-A1''} implies that for all $\tau \in (0, T]$,
$$
\left| A_2(\tau) - A_1(\tau) \right| \le \ds \frac{1}{9}M \tau^3 + \ds \frac{1}{2}|\delta\mu| \tau^2.
$$
Recall $T \le \min\{1, \tau_0\}$. It follows that for all $\tau \in (0, T]$,
\begin{align*}
\left| \ds \frac{A_2(\tau)}{A_1(\tau)} - 1 \right| \le{}& \ds \frac{\ds \frac{1}{9}M \tau^3 + \ds \frac{1}{2}|\delta\mu| \tau^2}{|A_1(\tau)|} \le \ds \frac{\ds \frac{1}{9}M \tau^3 + \ds \frac{1}{2}|\delta\mu| \tau^2}{\ds \frac{1}{2} \tau} = \ds \frac{2}{9}M \tau^2 + |\delta\mu| \tau \\
\le{}& \ds \frac{2}{9}M T^2 + |\delta\mu| T \le \ds \frac{2}{9}M + |\delta\mu|.
\end{align*}
By the definition of $M$, we have $M \le \ds \frac{2}{9}M + |\delta\mu|$, which implies
\be \label{ineq:M}
M \le \ds \frac{9}{7}|\delta\mu|.
\ee

When $\delta\mu = 0$ (i.e., $\mu_1 = \mu_2$), Inequality \eqref{ineq:M} requires $M=0$. Consequently, $A_2(\tau) = A_1(\tau)$ for all $\tau \in [0, T]$. This proves the uniqueness of the solution on the interval $[0, T]$. Recall that with $\mu_1 \in (\mu_0 - \delta, \mu_0 + \delta)$, $A_1(\tau) < 0$ for all $\tau \in [T, \tau_0]$ by our choice of $\tau_0$ and $\delta$. Then $A''' = \ds \frac{\tau A'}{6A}$ with the initial conditions $A(T) = A_1(T) < 0$, $A'(T) = A'_1(T)$, $A''(T) = A''_1(T)$ constitutes a regular initial value problem, to which $A_1(\tau)$ is the unique solution throughout the interval $[T, \tau_0]$ according to the standard existence and uniqueness theory. Altogether, we have shown the uniqueness of the solution on the full interval $[0, \tau_0]$.

Since the choice of $T$ is independent of $\mu_1$ and $\mu_2$, after combining \eqref{A2''-A1''} and \eqref{ineq:M} and integrating we can obtain a uniform constant $C_T > 0$ such that for any $\mu_1, \mu_2 \in (\mu_0 - \delta, \mu_0 + \delta)$,
$$
\left| A_2(\tau) - A_1(\tau) \right| + \left| A'_2(\tau) - A'_1(\tau) \right| + \left| A''_2(\tau) - A''_1(\tau) \right| \le C_T |\mu_2 - \mu_1|
$$
for all $\tau \in [0, T]$. Since $A_{1,2}(\tau) < q(\tau) < 0$ for all $\tau \in [T, \tau_0]$, by applying Gronwall's inequality on the interval $[T, \tau_0]$ we can extend the above inequality to the full interval $[0, \tau_0]$ after replacing $C_T$ with some larger $C > 0$.
\end{proof}

Let $A_0(\tau)$ be the solution to the IVP \eqref{IVP} with $A''_0(0) = \mu_0 \in \R$. Clearly, we can extend $A_0(\tau)$ to any $\bar\tau > \tau_0$ as long as $A_0(\tau) < 0$ for all $\tau \in (0, \bar\tau]$. Furthermore, by considering $A''' = \ds \frac{\tau A'}{6A}$ with the initial conditions specified at $\tau_0$ and set close to $(A_0(\tau_0), A'_0(\tau_0), A''_0(\tau_0))$, we can extend Lemma \ref{EUC} to the interval $[0, \bar\tau]$. In particular, by applying Gronwall's inequality on the interval $[\tau_0, \bar\tau]$, we obtain the following corollary.

\begin{cor} \label{dependence on mu}
For any $\bar\tau > 0$ such that $A_0(\tau) < 0$ for all $\tau \in (0, \bar\tau]$, there exist $\bar{C} > 0$ and a sufficiently small $\bar\delta > 0$ such that for any $\mu \in (\mu_0 - \bar\delta, \mu_0 + \bar\delta)$ the solution $A(\tau)$ with $A''(0) = \mu$ remains negative for all $\tau \in (0, \bar\tau]$ and satisfies
\be\label{A on mu}
\left| A(\tau) - A_0(\tau) \right| + \left| A'(\tau) - A'_0(\tau) \right| + \left| A''(\tau) - A''_0(\tau) \right| \le \bar{C} |\mu - \mu_0|
\ee
for all $\tau \in [0, \bar\tau]$.
\end{cor}

\subsection{General behavior of solutions}
Now that we know \eqref{IVP} is well-posed, we demonstrate 
some properties of solutions when $\mu \ge 0$.
First, its solutions reach a first minimum in finite time.
Let 
\be\label{tau1}
\ds \tau_1 := \sup \left\{\tau>0: \sup_{0<s<\tau} A'(s) < 0 \right\}.
\ee
Note that $\tau_1$ depends upon $\mu$, though we supress this for notational simplicity.

\begin{lemma}\label{first min}
For any $\mu \ge 0$ we have $0<\tau_1<\infty$. In addition,
 $A'(\tau_1)=0$ and $A''(\tau_1) > 0$.  Moreover there exist
 $\mu_0, a_0,  a_1>0$ such that
 $\mu \ge \mu_0$ implies (a) $\ds \frac{1}{2\mu} \le \tau_1 \le \ds\frac{1}{\mu}$, (b) 
 $-\ds \frac{a_0}{\mu} \le A(\tau_1) \le -\frac{a_1}{\mu}$ and (c)
 $\mu \le A''(\tau_1) \le 2 \mu$.
\end{lemma}

\begin{proof}
Clearly, $A'(\tau_1) = 0$ if $\tau_1$ is finite.  For $\tau \in I_1 := (0, \tau_1)$ (possibly $\tau_1 = \infty$), $A$ and $A'$ are negative and consequently $A''' > 0$.  Thus $A''$ is increasing.  This implies that $A''(\tau)> \mu \ge 0$ for $\tau \in I_1$ and also $A''(\tau_1) > 0$ if $\tau_1$ is finite.  Thus $A$ is concave up on $I_1$.  This together with the fact that $A(0) = 0$ implies via the mean value theorem that $|A(\tau)/\tau| > |A'(\tau)|$ for all $\tau \in I_1$ and so $\tau A'/A < 1$.  Thus we have 
\be \label{key} 
0 < A''' < 1/6
\ee 
for $\tau \in I_1$.  Integrating this inequality twice gives $\mu \tau - 1 \le  A' \le \tau^2/12 + \mu \tau - 1$ for all $\tau \in I_1$.  If $\mu > 0$, the lower bound on $A'$ becomes positive in finite time, implying that $\tau_1 < \infty$.  Examination of the zeroes of the upper and lower bounds gives the estimates for $\tau_1$ stated in the lemma.  Further integration of \eqref{key} gives the estimates on $A$ and $A''$.  If $\mu = 0$, the lower bound on $A'$ is not sufficient to conclude that $\tau_1 < \infty$.  However the fact that $A'''(0) = 1/6$ can be used to improve the lower bound in this case.  The details are standard and omitted.

\end{proof}

After reaching their first local minimum at $\tau = \tau_1$, solutions $A$ either reach zero in finite time
or alternately reach a local maximum (for which $A<0$) in finite time, as we now demonstrate.
Let
\be\label{tau2}
\ds \tau_2 := \sup \left\{ \tau>\tau_1: \inf_{\tau_1 < s< \tau} A'(s) \ge 0\ {\textrm{ and }} \sup_{\tau_1 < s< \tau} A(s) < 0\right\}.
\ee
As with $\tau_1$, $\tau_2$ implicitly depends upon $\mu$.
If this is finite, then we have either $A(\tau_2) = 0$ and $A'(\tau_2) \ge 0$ or $A(\tau_2) < 0$ and $A'(\tau_2) = 0$.  

\begin{lemma}
For all $\mu\ge0$ we have $\tau_1<\tau_2 < \infty.$
\end{lemma}

\begin{proof}
Clearly $\tau_2>\tau_1$ since $A''(\tau_1) > 0$.  
Suppose that $\tau_2 = \infty$.  Since $A$ is bounded above
and nondecreasing for $\tau > \tau_1$, we have $A(\tau)$ convergent as $\tau \to \infty$.
Since we have $A<0$ and $A' \ge 0$ for all $\tau > \tau_1$, \eqref{IVP} implies
$A''' \le 0$ for all $\tau > \tau_1$ as well.  Thus $A''$ is nonincreasing.  Since it
 is initially positive and $A$ converges, clearly $A''$ must eventually become negative.
 Since $A''$ is nonincreasing, once it is negative it remains negative for all subsequent times.
Thus there is an $\ep>0$ and a value $\tau_3>\tau_1$ such
that $A''(\tau) < -\ep$ for all $\tau>\tau_3$.  Integrating this, however, indicates that $A'$
must eventually become negative,  a contradiction, and the lemma is shown.
\end{proof}

%%%%%%%%%%%%%%%%%%%%%%%%%%%%%%%%%%%%%

\subsection{Energy estimates and behavior of $A$ for $\mu=0$ or $\mu$ large}
Our next goal is to show that for $\mu$ large, solutions go to zero in finite time but when $\mu=0$ the solution
has a local negative maximum at $\tau_2$.  This requires the following energy estimates, similar to those 
for \eqref{stoopid model}:
\begin{lemma}
\label{energy}
Fix $\mu$ and take $\tau_1$ and $\tau_2$ as above.
Let
$$
V_1 (a):=
{\tau_1 \over 6}
a \ln \left({a \over A(\tau_1)}\right) + \left(A''(\tau_1) - {\tau_1 \over 6} \right) \left(a-A(\tau_1)\right).
$$
Then
for all $\tau_1 < \tau \le \tau_2$
$$
{1 \over 2} [A'(\tau)]^2 < V_1(A(\tau)).
$$

Moreover, take $\bar{\tau} := 6 A''(\tau_1)$ and set
$$
V_2(a):=
{\bar{\tau} \over 6}
a \ln \left({a \over A(\tau_1)}\right).
$$
If $\bar{\tau} \ge \tau_2$ then for all $\tau_1 < \tau \le \tau_2$
$$
V_2(A(\tau)) < {1 \over 2} [A'(\tau)]^2.
$$
\end{lemma}

\begin{proof}
Let $I_2:=(\tau_1, \tau_2)$.  By definition of $\tau_2$, we have $A(\tau) < 0$
and $A'(\tau) \ge 0$ for all $\tau\in I_2$. If $\bar{\tau} \ge \tau_2$ then we have for all $\tau \in I_2$, 
$$
{\bar{\tau} A'(\tau) \over 6 A(\tau)} \le A'''(\tau) \le {\tau_1 A'(\tau) \over 6 A(\tau)},
$$
where the equalities hold only if $A'(\tau) = 0$ and the second inequality is true independent of $\bar{\tau}$. Integrating this from $\tau_1$ to $\tau$ gives
$$
A''(\tau_1) +{\bar{\tau} \over 6}\ln\left({A(\tau)\over A(\tau_1)}\right)
< A''(\tau) < 
A''(\tau_1)+ {\tau_1 \over 6}\ln\left({A(\tau)\over A(\tau_1)}\right)
$$
for all $\tau \in I_2$. Note that we obtain strict inequalities here because with $A''(\tau_1) > 0$ there is a small $\alpha > 0$ such that $A'(\tau) > 0$ for all $\tau \in (\tau_1, \tau_1+\alpha)$. If we multiply by $A'(\tau)$ and integrate once more from $\tau_1$ to any $\tau \in (\tau_1, \tau_2]$ then the conclusions of the lemma follow. In particular, $\bar{\tau}$ is irrelevant to the upper bound $V_1$.
\end{proof} 

\begin{lemma} \label{reach 0}
There exists $M>0$ such that $\mu\ge M$ 
implies $A(\tau_2)=0$ and $A'(\tau_2) > 0$.
%$A'(\tau) >0$ for $\tau_1<\tau<\tau_2$, $A'(\tau_2) \ge 0$ and $A(\tau_2)=0$.
\end{lemma}

\begin{proof}
First we show that if 
 $\bar{\tau} \ge \tau_2$ then we are done.  We know by the definition
of $\tau_2$ that either $A'(\tau_2)=0$ or $A(\tau_2)=0$.  Inspection of
$V_2(a)$ demonstrates that $V_2(a) > 0$ for $a \in (A(\tau_1),0)$ and 
is zero at the ends of that interval (see Figure~\ref{V2fig}).  Thus,  since $A$ is increasing on $(\tau_1,\tau_2)$
we have $A(\tau_2) \in (A(\tau_1),0]$ and
therefore $V_2(A(\tau_2)) \ge 0$.
We also know for $\tau \in (\tau_1,\tau_2)$
that $A' \ge 0$.  So in particular 
Lemma \ref{energy} implies $A'(\tau_2) > \sqrt{2 V_2(A(\tau_2))} \ge 0$. 
So, since $A'(\tau_2) \ne 0$ it must be the case that $A(\tau_2)=0$.
\begin{figure}
\includegraphics[width=0.9\textwidth]{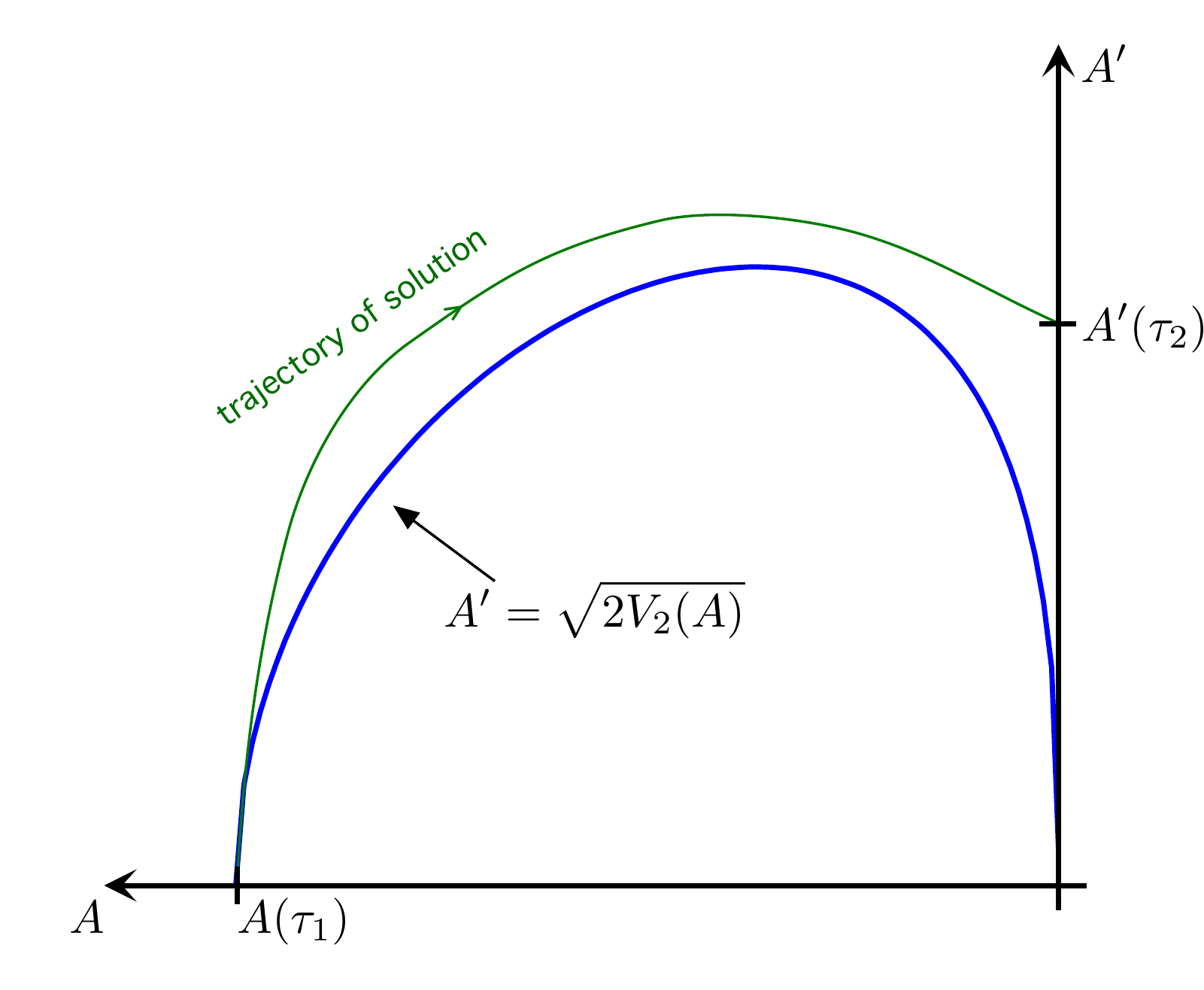}%
\caption{Sketch of $A'(\tau)$ vs $A(\tau)$ when $\mu$ is large.\label{V2fig}}
\end{figure}

On the other hand perhaps $\tau_2 > \bar{\tau}=6 A''(\tau_1)$.  
We will prove by contradiction that this is impossible when $\mu$ is large.
The estimate in the proof of Lemma \ref{energy}
gives:
$$
1 \le {A'(\tau) \over \sqrt{2 V_2(A(\tau))}}
$$
which is valid for $\tau_1 < \tau \le \bar{\tau}$ if we assume $\bar{\tau} < \tau_2$.  (Note that for $\mu$ big enough, the estimates on $\tau_1$ in Lemma \ref{first min} imply that $\bar{\tau} > \tau_1$.)
Integration from $\tau_1$ to $\bar{\tau}$ gives
$$
\bar{\tau}-\tau_1 \le 
  \int_{\tau_1}^{\bar{\tau}} {A'(s) ds \over \sqrt{2 V_2(A(s))}}
  = 
 \int_{A(\tau_1)}^{A(\bar{\tau})} {da \over \sqrt{2 V_2(a)}}.
$$
Substitution in from the definition of $V_2$ and $\bar{\tau}$ gives
$$
6 A''(\tau_1) -\tau_1 
\le 
{1 \over \sqrt{2 A''(\tau_1)}}\int_{A(\tau_1)}^{A(\bar{\tau})}
{d a \over \sqrt{ a \ln\left( {a / A(\tau_1)} \right)} }.
$$
A change of variables and the fact that $1>A(\bar{\tau})/A(\tau_1)>0$
implies:
$$
6 A''(\tau_1) -\tau_1 \le 
\sqrt{ |A(\tau_1)| \over 2 A''(\tau_1)}\int^{1}_{A(\bar{\tau})/A(\tau_1)}
{d b \over \sqrt{ \left \vert b \ln(b)  \right\vert} }
\le 
\sqrt{ |A(\tau_1)| \over 2 A''(\tau_1)}\int^{1}_{0}
{d b \over \sqrt{ \left \vert b \ln (b)  \right\vert } }.
$$
Letting $\ds K=\int^{1}_{0}
{d b \over \sqrt{\left\vert b \ln( b)  \right\vert} }<\infty$ we have
$$
6A''(\tau_1) - \tau_1 \le K \sqrt{ |A(\tau_1)| \over 2 A''(\tau_1)}.
$$
Now we employ the estimates in Lemma \ref{first min}
to find for all $\mu \ge \mu_0$:
$$
6 \mu - {1 \over \mu} \le K \sqrt{{ a_0 / \mu \over 2 \mu} }= {K \over \mu} \sqrt{a_0 \over 2}.
$$
For $\mu$ sufficiently large this is impossible and so the lemma is shown.

\end{proof}
%%%%%%%%%%%%%%%%%%%%%%%%%%%%%%%%%%%%%

%\begin{lemma}
%Fix $\mu$ and take $\tau_1$ as in the previous Lemma.  Set
%$$
%V(a) := -{\tau_1\over 6} a \left(\ln (-a)-1\right) - \left( A''(\tau_1) - {\tau_1\over 6} \ln (-A(\tau_1))\right)a
%$$
%and 
%$$
%E(\tau):={1 \over 2} (A'(\tau))^2 + V(A(\tau)).
%$$
%Then
%$$
%{dE \over d\tau} = A'(\tau) Q(\tau)
%$$
%where
%$$
%Q(\tau) := {\tau - \tau_1 \over 6} \ln (-A(\tau)) - {1 \over 6} \int_{\tau_1}^\tau \ln (-A(s)) ds.
%$$
%\end{lemma}

%\begin{proof}

%Observe that $\eqref{IVP}$ can be rewritten as
%$$
%A''' = \partial_\tau \left ({\tau \over 6} \ln(-A) \right) - {1 \over 6} \ln(-A).
%$$
%Integrating this from $\tau_1$ to $\tau$ gives
%$$
%A'' ={ \tau \over 6} \ln(-A) + A''(\tau_1) - {\tau_1 \over 6} \ln(A(-\tau_1)) -{1 \over 6} \int_{\tau_1}^\tau \ln (-A(s)) ds .
%$$
%Computation of $dV/da$ and substitution into 
%$\ds
%{dE \over d\tau} = A' \left( A'' + {dV \over da}(A) \right)
%$
%gives the result.
%\end{proof}

\begin{lemma} \label{below 0}
If $\mu = 0$, then $A(\tau_2)<0$ and $A'(\tau_2)=0$.
\end{lemma}

\begin{figure}
\includegraphics[width=0.9\textwidth]{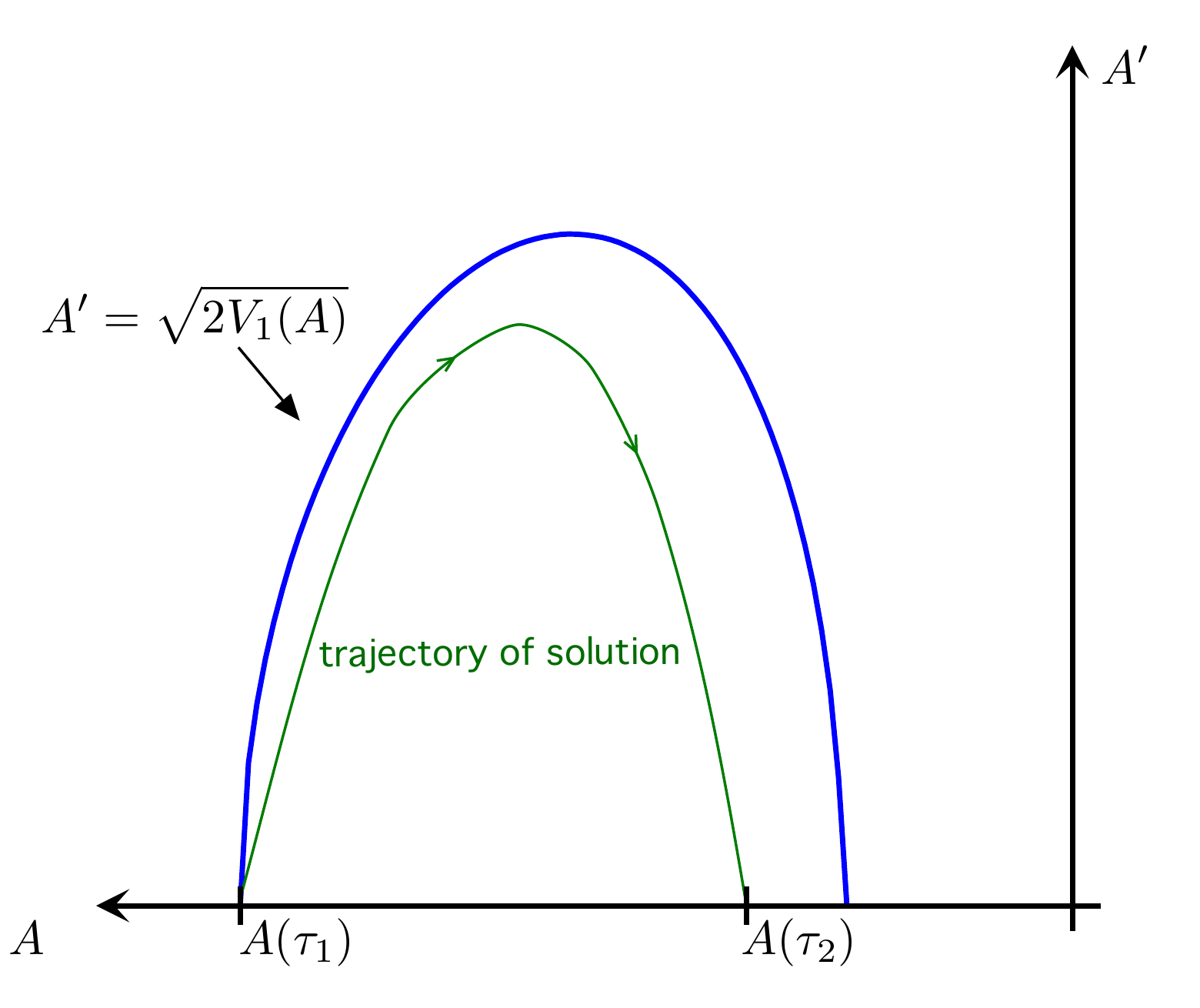}%
\caption{Sketch of $A'(\tau)$ vs $A(\tau)$ when $\mu=0$.\label{V1fig}}
\end{figure}

\begin{proof}
Proving the lemma amounts to showing that $A(\tau_2) \ne 0$. 
Notice that  $V_1(0) = A(\tau_1) \left(\tau_1/6-A''(\tau_1) \right)$.
Let $\rho(\tau) = \tau/6 - A''(\tau)$.  $\rho(0) = 0$, since $\mu = 0$.  Also $\rho'(\tau) = 1/6 - A'''(\tau)$.  In the proof of Lemma \ref{first min} we saw that $A''' < 1/6$ for $\tau \in I_1$ (see \eqref{key}) and so $\rho' > 0$. Consequently $\tau_1/6 - A''(\tau_1) > 0$.  Thus, since $A(\tau_1)<0$, we have
$
V_1(0) < 0.
$
Moreover,  $V_1(A(\tau_1)) = 0$ and $V'_1(A(\tau_1))=A''(\tau_1)>0$.
Thus there is a point $k \in (A(\tau_1),0)$ at which $V_1(k)=0$.  See Figure~\ref{V1fig}.

Since for $\tau \in [\tau_1,\tau_2]$ we have $0 \ge A(\tau) \ge A(\tau_1)$ and
$0 \le A'(\tau) \le \sqrt{2 V_1(A(\tau))}$ we must therefore have
$0>k\ge A(\tau)$.  Since $\tau_2$ is finite for $\mu = 0$ and
$A$ is bounded from zero for all $\tau_1 < \tau <\tau_2$, the definition \eqref{tau2} of $\tau_2$ 
requires $A'(\tau_2) = 0$.

\end{proof}

\subsection{Openness of qualitative behavior}
The next pair of results show that slightly varying $\mu$ does not change the qualitative behavior of $A$ at $\tau_2$.  That is to say,
if $\mu_0$ is such that $A(\tau_2)=0$ then there is an open neighborhood of $\mu_0$ where the corresponding solutions enjoy this same property.  Likewise, if $\mu_0$ such that $A'(\tau_2) = 0$, then there is an open neighborhood of $\mu_0$ where the corresponding solutions do the same.
Let $\hA(\tau)$ be the solution to the IVP \eqref{IVP} with $\hA''(0) = \hat{\mu} \ge 0$, and define $\htau_2$ for $\hA(\tau)$ as in \eqref{tau2}.

\begin{lemma} \label{open: reach 0}
Suppose $\hA(\htau_2) = 0$ and $\hA'(\htau_2) = 2h > 0$. Then for any $\epsilon > 0$ there exists a $\delta > 0$ such that for any $\mu \in (\hat{\mu} - \delta, \hat{\mu} + \delta)$, the solution $A(\tau)$ to the IVP \eqref{IVP} with $A''(0) = \mu$ reaches zero at $\tau_2 < \htau_2 + \epsilon$.
\end{lemma}

\begin{proof}
Since $\hA(\htau_2) = 0$, the right hand side of \eqref{ODE} is undefined at this point and so we cannot directly use our results about
continuous dependence on $\mu$ from Corollary \ref{dependence on mu} to prove this result.  Instead we will have to rely instead on the
nearly hamiltonian structure of \eqref{ODE}.

Integrating the ODE of \eqref{IVP} with the initial conditions $A(0) = 0$, $A'(0) = -1$, and $A''(0) = \mu$, we obtain that 
\begin{equation} \label{Int}
A(\tau) A''(\tau) - \ds \frac{1}{2} \left(A'(\tau)\right)^2 + \ds \frac{1}{2} = \ds \frac{1}{6} \tau A(\tau) - \ds \frac{1}{6} \int_0^{\tau} A(s)\, ds
\end{equation}
as long as $A < 0$ on the interval $(0, \tau)$. 

Since $\hA(\htau_2) = 0$ and $\hA'(\htau_2) =2h > 0$, we can choose a small $\alpha > 0$ such that $|\hA(\tau)| \ge |h (\tau - \htau_2)|$ and $\hA'(\tau) < 3h$ for all $\tau \in [\htau_2 - \alpha, \htau_2]$. It follows that 
$$
\left| \hA''(\tau) - \hA''( \htau_2 - \alpha ) \right| \le \int_{\htau_2 - \alpha}^{\tau} \left| \ds \frac{s \hA'(s)}{6 \hA(s)} \right| ds \le \int_{\htau_2 - \alpha}^{\tau} \left| \ds \frac{\htau_2 3h}{6 h (s - \htau_2)} \right| ds = \ds \frac{1}{2} \htau_2 \ln{ \ds \frac{\alpha}{\htau_2 - \tau} }.
$$
Therefore, $\hA(\tau) \hA''(\tau) \rightarrow 0$ as $\tau \rightarrow \htau_2^-$. Applying \eqref{Int} to $\hA$ and taking the limits of both sides as $\tau \rightarrow \htau_2^-$, we obtain
$$
-2h^2 + \ds \frac{1}{2} = -\ds \frac{1}{6} \int_0^{\htau_2} \hA(s)\, ds.
$$

Suppose there exist an $\epsilon > 0$ and a sequence $\{ \mu_n \} \rightarrow \hat{\mu}$ such that for any $n$, the solution $A_n(\tau)$ to the IVP \eqref{IVP} with $A_n''(0) = \mu_n$ is below zero for all $\tau \in (0, \htau_2 + \epsilon)$. Since $\hA(\tau) < 0$ for all $\tau \in (0, \htau_2)$, we can apply \eqref{A on mu} of Corollary \ref{dependence on mu} to obtain a subsequence $\{ \mu_{n_k} \}$ and some $Q > 0$ such that for each $k > Q$,
$$
|A_{n_k}(\tau) - \hA(\tau)| + |A'_{n_k}(\tau) - \hA'(\tau)| + |A''_{n_k}(\tau) - \hA''(\tau)| < \ds \frac{1}{k}
$$
for all $\tau \in [0, \htau_2 - \ds \frac{1}{k}]$. Clearly, $A'_{n_k}(\htau_2 - \ds \frac{1}{k}) \rightarrow \hA'(\htau_2) = 2h$. Without loss of generality, we assume that $A'_{n_k}(\htau_2 - \ds \frac{1}{k}) > h$ for all $k > Q$. Let
$$
\omega_k := \sup \left\{ \tau \in (\htau_2 - \ds \frac{1}{k}, \htau_2 + \epsilon): \inf_{\htau_2 - \ds \frac{1}{k} < s < \tau} A'_{n_k}(s) \ge h \right\} - \htau_2.
$$
Note that $\omega_k \rightarrow 0$ as $k \rightarrow \infty$. Otherwise, for some very large $k$, $A_{n_k}(\tau)$ would reach zero at some $\tau \in (\htau_2 - \ds \frac{1}{k}, \htau_2 + \omega_k)$ since $A_{n_k}(\htau_2 - \ds \frac{1}{k})$ converges to $\hA(\htau_2) = 0$ from below. Without loss of generality, we assume that $\htau_2 + \omega_k < \htau_2 + \epsilon$ for all $k > Q$. Then the definition of $\omega_k$ implies $A'_{n_k}(\htau_2 + \omega_k) = h$ and $A''_{n_k}(\htau_2 + \omega_k) \le 0$ for all $k > Q$. Applying \eqref{Int} to $A_{n_k}$ at $\tau = \htau_2 + \omega_k$ yields
$$
A_{n_k}(\htau_2 + \omega_k) A''_{n_k}(\htau_2 + \omega_k) - \ds \frac{1}{2} h^2 + \ds \frac{1}{2} = \ds \frac{1}{6} (\htau_2 + \omega_k) A_{n_k}(\htau_2 + \omega_k) - \ds \frac{1}{6} \int_0^{\htau_2 + \omega_k}\! A_{n_k}(s)\, ds.
$$
Since $0 > A_{n_k}(\htau_2 + \omega_k) > A_{n_k}(\htau_2 - \ds \frac{1}{k}) \rightarrow \hA(\htau_2) = 0$, the right-hand side converges to $-\ds \frac{1}{6} \int_0^{\htau_2} \hA(s)\, ds = -2h^2 + \ds \frac{1}{2}$. However, for all $k > Q$ the left-hand side is greater than or equal to $- \ds \frac{1}{2} h^2 + \ds \frac{1}{2}$. This is a contradiction.
\end{proof}

In the next lemma, we still assume $\hA''(0) = \hat{\mu} \ge 0$.

\begin{lemma} \label{open: below 0}
Suppose $\hA(\htau_2) < 0$ and $\hA'(\htau_2) = 0$. Then there exists a $\delta > 0$ such that for any $\mu \in (\hat{\mu} - \delta, \hat{\mu} + \delta)$, the solution $A(\tau)$ to the IVP \eqref{IVP} with $A''(0) = \mu$ satisfies $A(\tau_2) < 0$ and $A'(\tau_2) = 0$.
\end{lemma}

\begin{proof}
Define $\htau_1$ for $\hA(\tau)$ as in \eqref{tau1}. Then $\hA''(\htau_1) > 0$ since $\hA''(0) = \hat{\mu} \ge 0$. It follows that $\hA'(\htau_1) = \hA'(\htau_2) = 0$ and $\hA'(\tau) > 0$ for all $\tau \in (\htau_1, \htau_2)$. Thus there is a $\htau_c \in (\htau_1, \htau_2)$ such that $\hA''(\htau_c) = 0$. Furthermore, since $\hA''' = \ds \frac{\tau \hA'}{6 \hA} < 0$ for all $\tau \in [\htau_c, \htau_2)$, we have $\hA''(\htau_2) < 0$. This allows us to choose a small $\alpha > 0$ such that $\hA(\tau) < 0$ for all $\tau \in (0, \htau_2 + \alpha]$, $\hA'(\htau_2 + \alpha) < 0$, and $\hA''(\htau_2 + \alpha) < 0$. Take a small $\beta > 0$. Then Corollary \ref{dependence on mu} guarantees the existence of a $\delta > 0$ such that for any $\mu \in (\hat{\mu} - \delta, \hat{\mu} + \delta)$, the solution $A(\tau)$ with $A''(0) = \mu$ is below zero for all $\tau \in (0, \htau_2 + \alpha]$ and satisfies $A'(\htau_2 - \beta) > 0$ and $A'(\htau_2 + \alpha) < 0$. This implies that $A(\tau_2) < 0$ and $A'(\tau_2) = 0$ at $\tau_2 \in (\htau_2 - \beta, \htau_2 + \alpha)$.
\end{proof}

\subsection{Final steps}
We are now in a position to complete the proof of Proposition~\ref{ODE Prop}.
\begin{proof}
Define the set $\Omega$ as follows:
$$
\Omega := \left\{ \bar{\mu} \in \R : A(\tau_2) = 0 \text{ and } A'(\tau_2) > 0 \text{ for any } \mu \ge \bar{\mu} \right\}.
$$
Lemma \ref{reach 0} guarantees that $\Omega$ is nonempty, and Lemma \ref{below 0} shows that zero is a lower bound of $\Omega$. Take $\mu_* := \inf \Omega > 0$, with the inequality guaranteed by Lemma \ref{below 0} and Lemma \ref{open: below 0}. Let $A_*(\tau)$ be the solution to the IVP \eqref{IVP} with $A_*''(0) = \mu_*$. Then by Lemma \ref{open: reach 0} and the definition of $\mu_*$, it is only possible that either $A_*(\tau_2) = 0$ and $A_*'(\tau_2) = 0$ or $A_*(\tau_2) < 0$ and $A_*'(\tau_2) = 0$. Furthermore, we rule out the second case by the definition of $\mu_*$ and Lemma \ref{open: below 0}. This proves (i)--(iv) of Proposition \ref{ODE Prop}.

Since $A_*(\tau) < 0$ and $A_*'(\tau) > 0$ for all $\tau \in (\tau_1, \tau_2)$, we have 
$$
\ds \frac{\tau_2 A_*'}{6A_*} < A_*''' < \ds \frac{\tau_1 A_*'}{6A_*}
$$
for all $\tau \in (\tau_1, \tau_2)$. Integrating the above inequality gives
$$
\ds \frac{\tau_2}{6} \ln \left| \ds \frac{A_*(\tau)}{A_*(\tau_1)} \right| + A_*''(\tau_1) < A_*''(\tau) < \ds \frac{\tau_1}{6} \ln \left| \ds \frac{A_*(\tau)}{A_*(\tau_1)} \right| + A_*''(\tau_1)
$$
for all $\tau \in (\tau_1, \tau_2)$. Thus $A_*''(\tau) \rightarrow -\infty$ as $\tau \rightarrow \tau_2^-$. Then we can take a sufficiently small $\alpha > 0$ such that $A_*''(\tau) < -1$ for all $\tau \in (\tau_2 - \alpha, \tau_2)$. By the mean value theorem we have
$$
A_*'(\tau_2) - A_*'(\tau) = A''(s) (\tau_2 -\tau) < - (\tau_2 -\tau)
$$
for any $\tau \in (\tau_2 - \alpha, \tau_2)$ and some $s \in (\tau, \tau_2)$. Since $A_*'(\tau_2) = 0$, this implies that $0 < (\tau_2 -\tau) < A_*'(\tau)$ and consequently $\ds \frac{1}{2} (\tau_2 -\tau)^2 < A_*(\tau_2) - A_*(\tau) = -A_*(\tau)$ for all $\tau \in (\tau_2 - \alpha, \tau_2)$. Thus $\ln|A_*(\tau)| = O(\ln(|\tau-\tau_2|))$ as $\tau \to \tau_2^-$. This proves (v) of Proposition \ref{ODE Prop}.

\end{proof}

\section{Numerical Study of $K(2,2)$}
\label{SC}

\subsection{Regularization and scaling}
In this section we numerically assess the ill-posedness of
\eqref{K22}.  But it is inappropriate to directly simulate 
an equation
that not only lacks a local well-posedness theory but for which we
suspect ill-posedness.  Thus we regularize the equation.  Simulating
this regularized problem, we find evidence of the ill-posedness as we
let the regularization parameter vanish.

We study the following regularization of \eqref{K22}
\begin{equation}
\label{e:knm_cons_reg}
(I + \delta \partial_x^4) \partial_t u = \partial_x \paren{u^2} + \partial_x^3 \paren{u^2}.
\end{equation}
Here $I$ is the identity and $\delta > 0$ is the small regularization parameter.
We choose this particular regularization since its implementation is natural when simulating
solutions of \eqref{K22} {via} a pseudospectral method.

  Inverting the
operator on the left-hand side to be able to write it as an evolution
equation, we see that $\partial_x/(I + \delta \partial_x^4)$
and $\partial_x^3/(I + \delta \partial_x^4)$ are bounded operators.
Indeed,
$$
\left\|{\frac{\partial_x}{I + \delta \partial_x^4}}\right\|_{H^{s} \to H^{s+3}}
\le C \delta^{-1} \quad {\textrm{and}} \quad
\left\|{\frac{\partial_x^3}{I + \delta \partial_x^4}}\right\|_{H^{s} \to H^{s+1}}
\le C\delta^{-1}.
$$
The boundedness of these operators make it trivial to prove:
\begin{theorem}[Local Well-Posedness of a Regularized Problem]
\label{thm:lwp_kmn_cons_reg}
Fix
$\delta > 0$.  Then \eqref{e:knm_cons_reg} is locally well-posed in
$H^s$ for any $s>\tfrac{1}{2}$
\begin{itemize}
\item For all $u_0\in H^s$, there exists a $T>0$ and a unique function $u(t) \in
C(0,T;H^s)$ solving the integral form of \eqref{e:knm_cons_reg},

\item The solution will depend continuously upon the data,
\item There exists a maximal time of existence, $T_{\exist}$ such that if it is 
finite, 
\[
\lim_{t\to T_{\exist}} \norm{u(t)}_{H^s} = \infty
\]
\end{itemize}
\end{theorem}
The proof of this, which we omit, is an elementary application of the
fixed point technique.  The operator is bounded in $L^2$-based Sobolev spaces.  For $s>\tfrac{1}{2}$, $H^s$ is an algebra which
makes the nonlinearity easy to treat.
An impediment to extending Theorem \ref{thm:lwp_kmn_cons_reg} to one which is
global in time, or even one which holds on time intervals of a length
uniform in $\delta$, 
is the lack of an obvious coercive conserved
quantity associated with \eqref{e:knm_cons_reg}. 
The only obvious conserved quantity for \eqref{e:knm_cons_reg}
is
$$
\int u(x,t) dx = \int u(x,0) dx.
$$
This is formally an invariant for solutions of $\eqref{K22}$ as well.

%However, a small
%data result may exist.
%{\bf note from gs: the above theorem can probably be generalized
%  significantly, both to lower regularity sobolev spaces with integer
%  nonlinearities, and non integer nonlinearities, provided they are
%  sufficiently smooth, and we work with integer valued sobolev
%  spaces.  However, this is sufficient for what we are after}

By simulating the regularized equation \eqref{e:knm_cons_reg} we
seek evidence of ill-posedness of \eqref{K22}.  In
particular, we shall find a sequence of vanishing initial conditions for
\eqref{e:knm_cons_reg}, such that as $\delta \to 0$, the corresponding
solutions at $t = T_\star>0$ have $H^2$ norm of (at least) unit size.
To that end, we will scale the initial data for \eqref{e:knm_cons_reg}
using the same scaling given by \eqref{DDEscaling} and link the vanishing of 
the smoothing parameter $\delta$ to the scaling parameter $\lambda$. Specifically we will
study \eqref{e:knm_cons_reg} with
\be \label{params}
\delta = \delta_\lambda:=.1 \lambda^4 {\quad \textrm{and} \quad} u(x,0) = f_\lambda(x) :=  \lambda f \left({ x \over \lambda^{1/3}}\right)
\ee
where
$$
f (x) = \exp(-4x^2)
$$
and $\lambda \to 0$.  Note that the calculations leading to \eqref{sobscaling}
show that $\| f_\lambda \|_{H^2} \le C \lambda^{7/6}$ and so this choice of initial
data vanishes in the limit.
We denote the solution of \eqref{e:knm_cons_reg} with the choices 
in \eqref{params} by $u_\lambda(x,t)$.  Note that we have taken our initial data to 
be everywhere positive, unlike the self-similar solutions found for \eqref{DDE}.  
We do this to demonstrate that the problematic effects of degeneracy manifest themselves
even for solutions which do not cross the $x$-axis.  

The choices for $\delta_\lambda$ and $f_\lambda(x)$ are made for the following reason.
Instead of scaling just the initial data for \eqref{e:knm_cons_reg} suppose that we 
rescale the whole equation by
\begin{equation}
\label{e:u_scaling}
u_\lambda(x,t) = \lambda v_\lambda \left(  {x \over \lambda^{1/3}}, t \right).
\end{equation}
Then \eqref{e:knm_cons_reg} becomes:
$$
(I + \delta_\lambda \lambda^{-4/3} \partial_y^4) \partial_t v_\lambda = \lambda^{2/3} 
 \partial_y v_\lambda^2 + \partial_y^3 v_\lambda^2
$$
where $y = x/\lambda^{1/3}$.  Plugging in from \eqref{params} we get
$$
v_\lambda(y,0) = f(y)
$$
and
\begin{equation}
\label{e:knm_cons_resacled}
(I + .1 \lambda^{8/3} \partial_y^4) \partial_t v_\lambda = \lambda^{2/3} 
 \partial_y v_\lambda^2 + \partial_y^3 v_\lambda^2.
\end{equation}

Of course \eqref{e:knm_cons_resacled} is not an exact scaling. Notice that
as $\lambda \to 0$ the term $\partial_y v_\lambda^2$ term will vanish.  The equation
will asymptotically be dominated by the third derivative term which is
consistent with our expectation that this term is the source of the
ill-posedness.  Moreover notice that the coefficient of regularization even in these scaled coordinates
vanishes.  That is to say, we have chosen the regularization parameter so that it  vanishes ``more rapidly"
than scaling effects.

\subsection{Computational results}

%Using the scalings \eqref{e:u_scaling} and  \eqref{e:delta_scaling},
%we aim to construct an ensemble $(f_\lambda, \delta_\lambda)$ of
%intial conditions and regularization parameters, such that for some
%range of $s\geq 0$:
%\begin{equation}
%\lim_{\lambda \to 0} \norm{f_\lambda}_{H^s} = 0, \quad \lim_{\lambda \to 0}
%\delta_\lambda = 0.
%\end{equation}
%But if we merely use these scalings, then \eqref{e:hs_rescaled}
%predicts that the subcritical Sobolev norms will vanish with
%$\lambda$.  Thus, we must send $\delta_\lambda \to 0$ {\it faster}
%than demanded by \eqref{e:delta_scaling}.  

We integrate our problem pseudospectrally with a Crank-Nicholson time
stepping scheme.  Some additional details are presented in Section \ref{s:numerics}.
%\subsubsection{Initial Conditions}
%
%We now restrict our attention to the $K(2,2)$ equation.  As our
%principle initial condition and regularization parameter, we take
%\begin{equation}
%\label{e:principle_ic}
%f(x) = \exp(-4 x^2), \quad \delta = .1
%\end{equation}
%We then generate the $(f_\lambda, \delta_\lambda)$ pairs with
%\begin{equation}
%\label{e:ic_scaled}
%f_\lambda(x) = \lambda f(\lambda^{-1/3} x), \quad \delta_\lambda =
%\lambda^4 \delta
%t\end{equation}
%\subsubsection{Evolution of the Data}
Simulating \eqref{e:knm_cons_reg} with \eqref{params} to
$T=.1$ on the domain $[-2\pi, 2\pi)$ with $8192$ resolved Fourier
modes, the data evolves as in Figures \ref{f:evolution_lam04} and \ref{f:evolution_lam005}.  At large
values of $\lambda$, there is some development of oscillations.  At
very small values of $\lambda$, it develops a highly oscillatory
structure. Note that these oscillations appear for $x>0$, where $f' < 0$.  That is to
say in the exactly the place where the term $6 u_x u_{xx}$ in \eqref{K22} acts, heuristically,
as a backwards heat operator.

\begin{figure}
\subfigure[]{\includegraphics[width=2.2in]{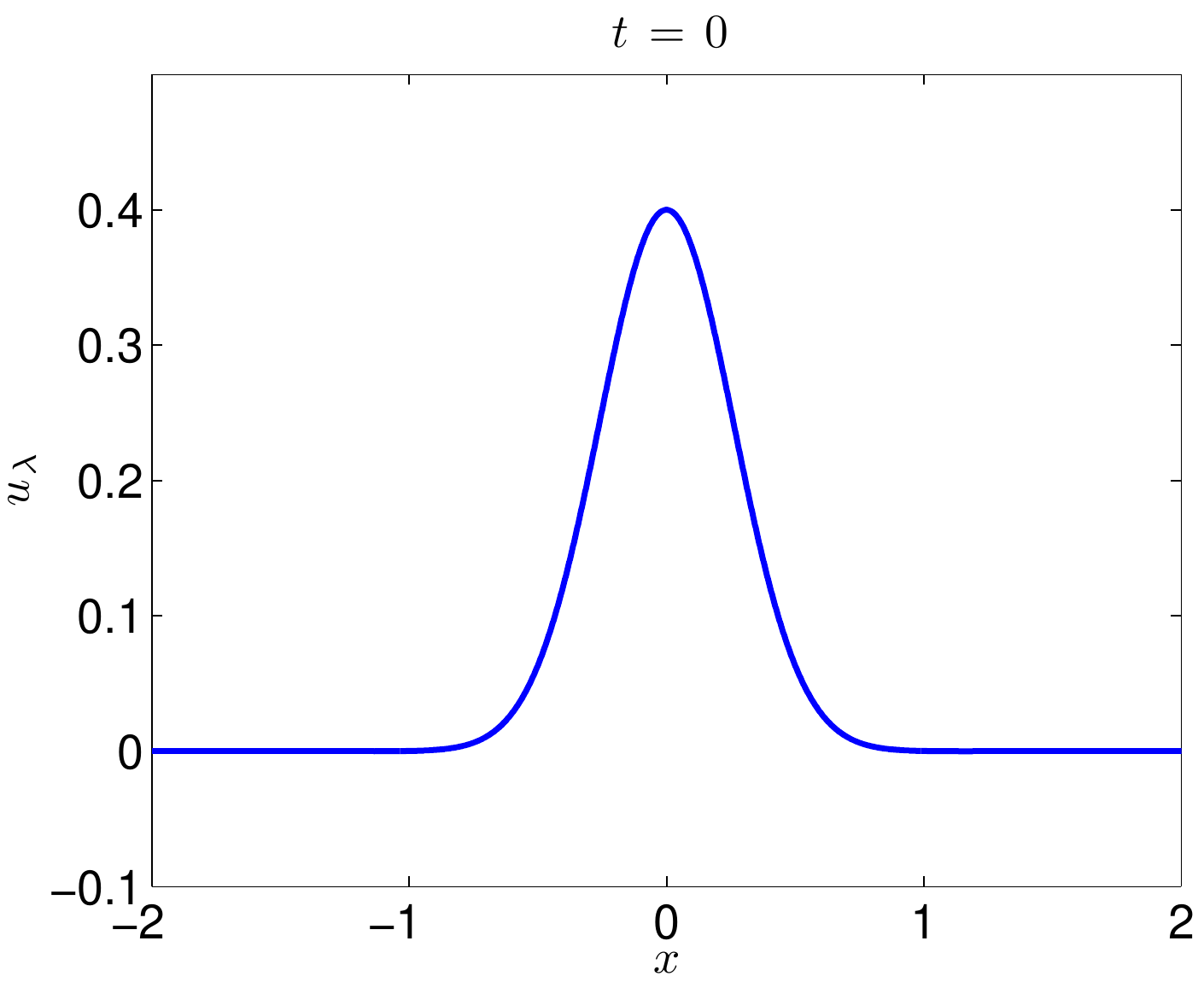}}
\subfigure[]{\includegraphics[width=2.2in]{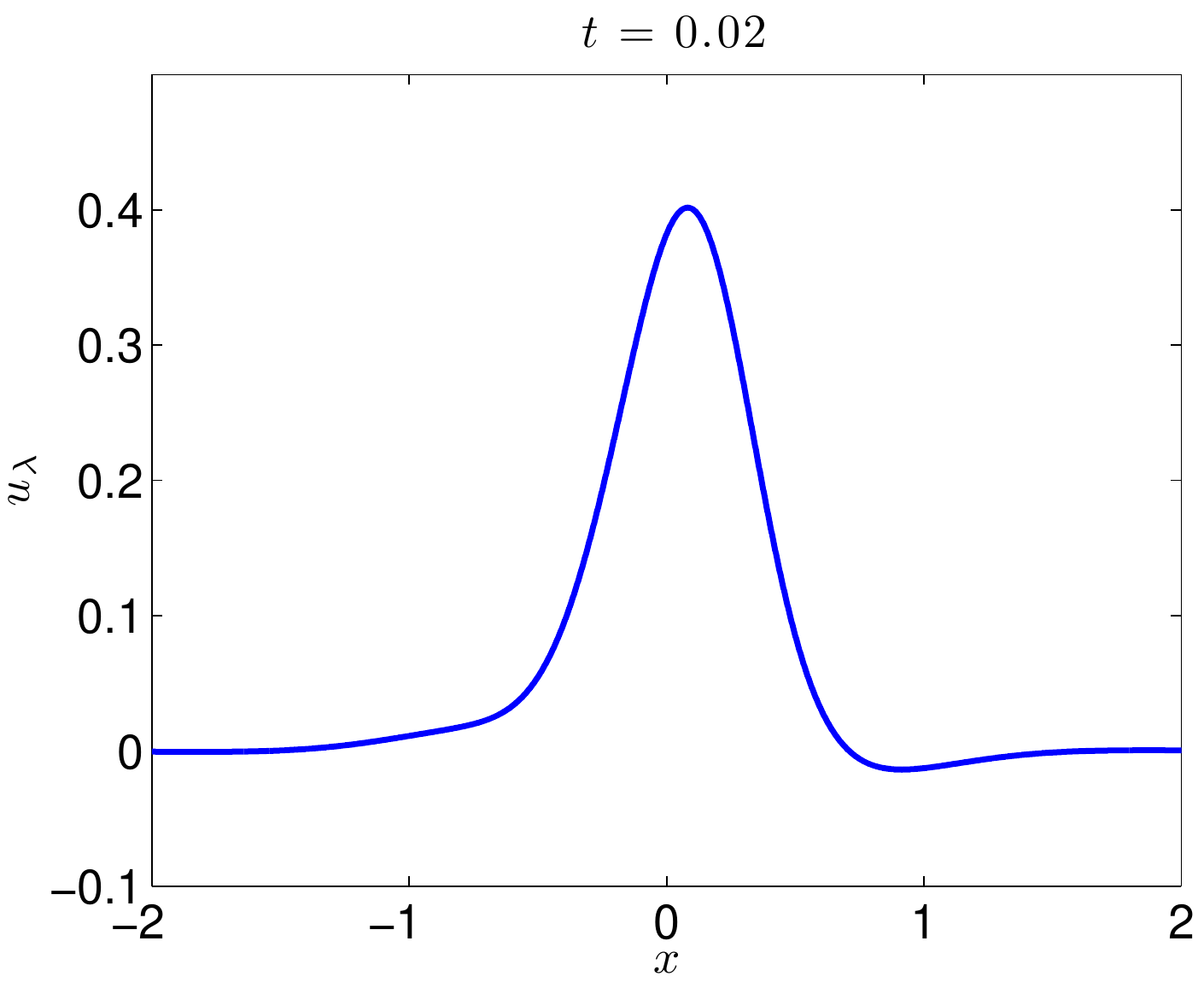}}
\subfigure[]{\includegraphics[width=2.2in]{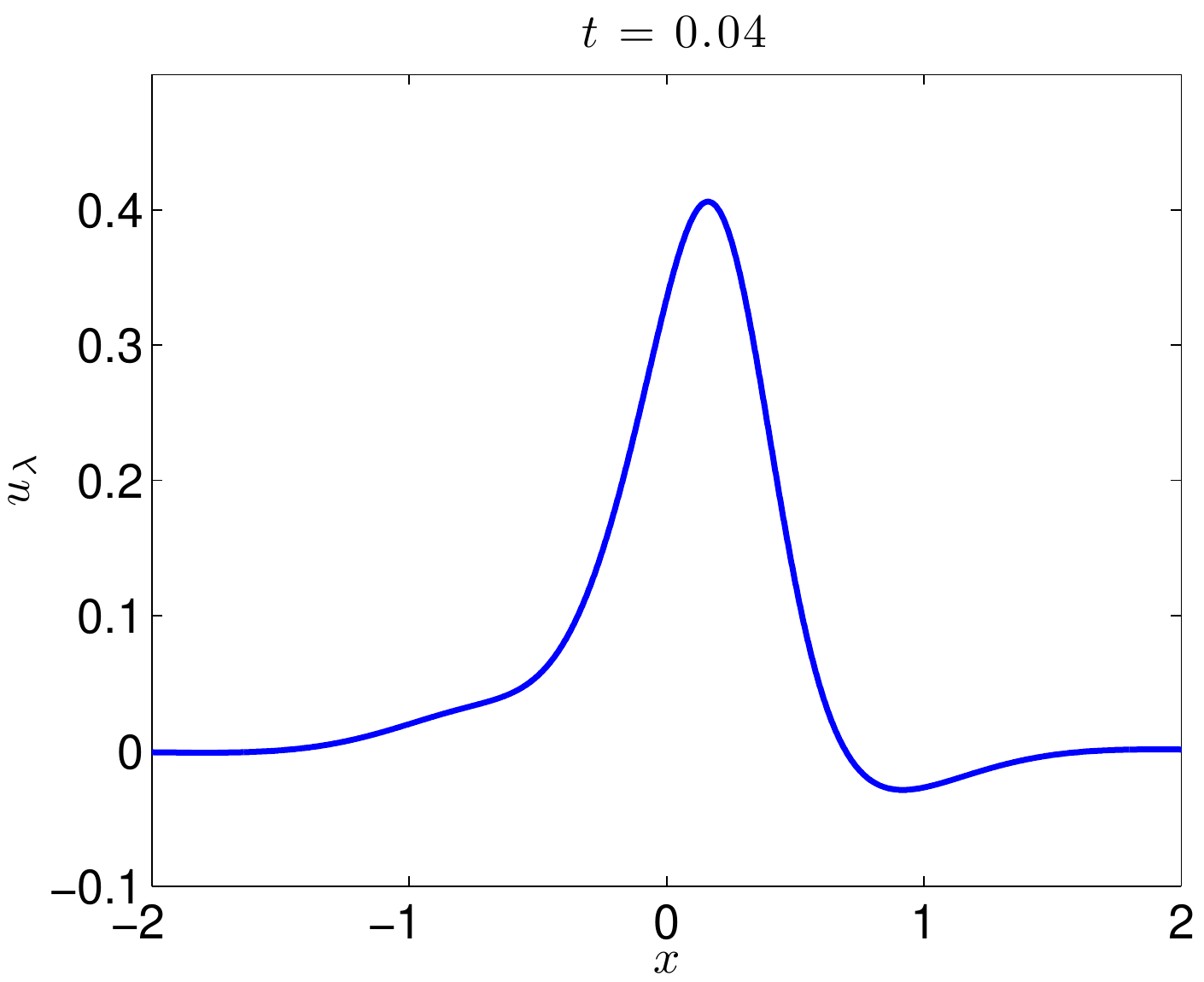}}
\subfigure[]{\includegraphics[width=2.2in]{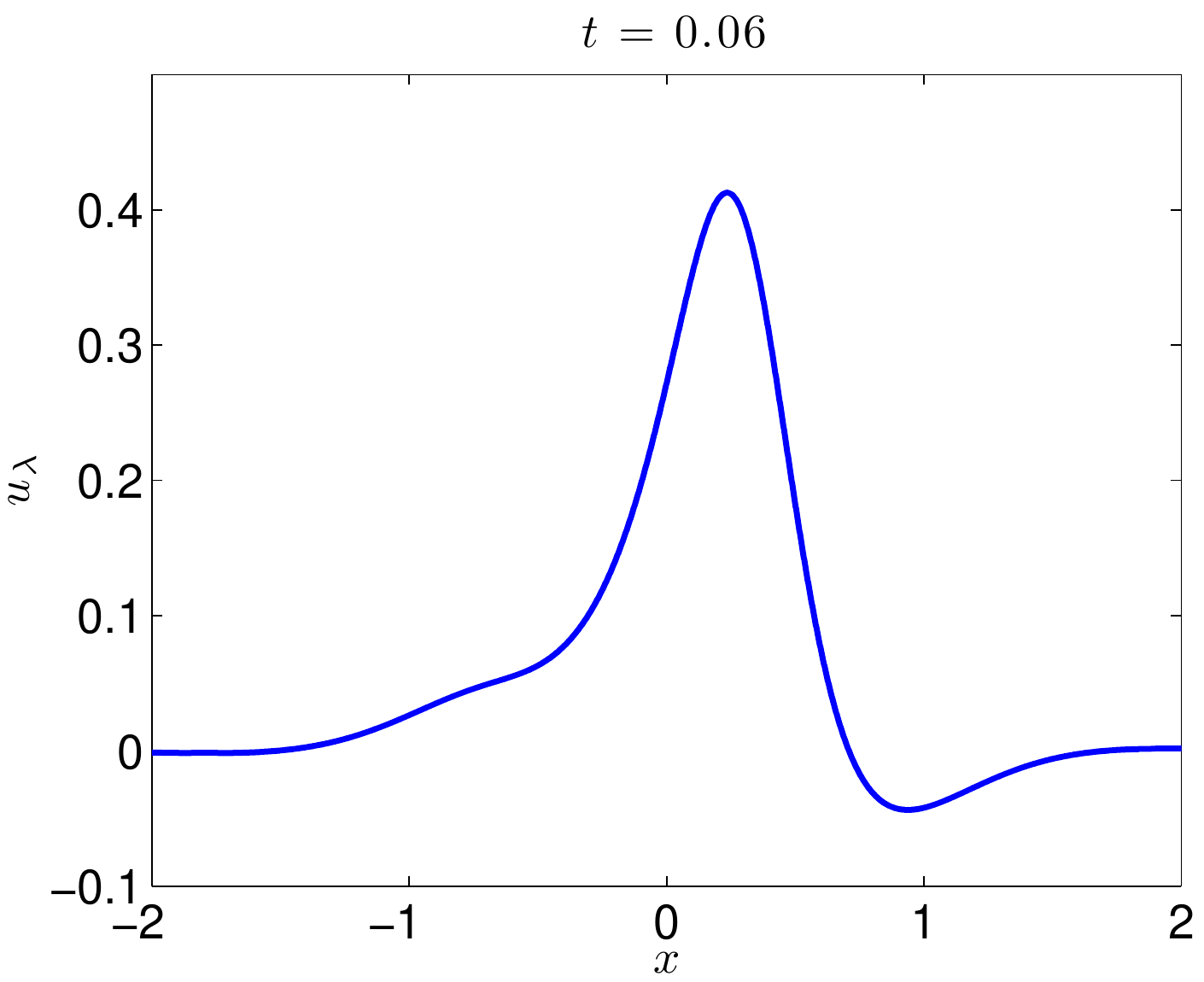}}
\subfigure[]{\includegraphics[width=2.2in]{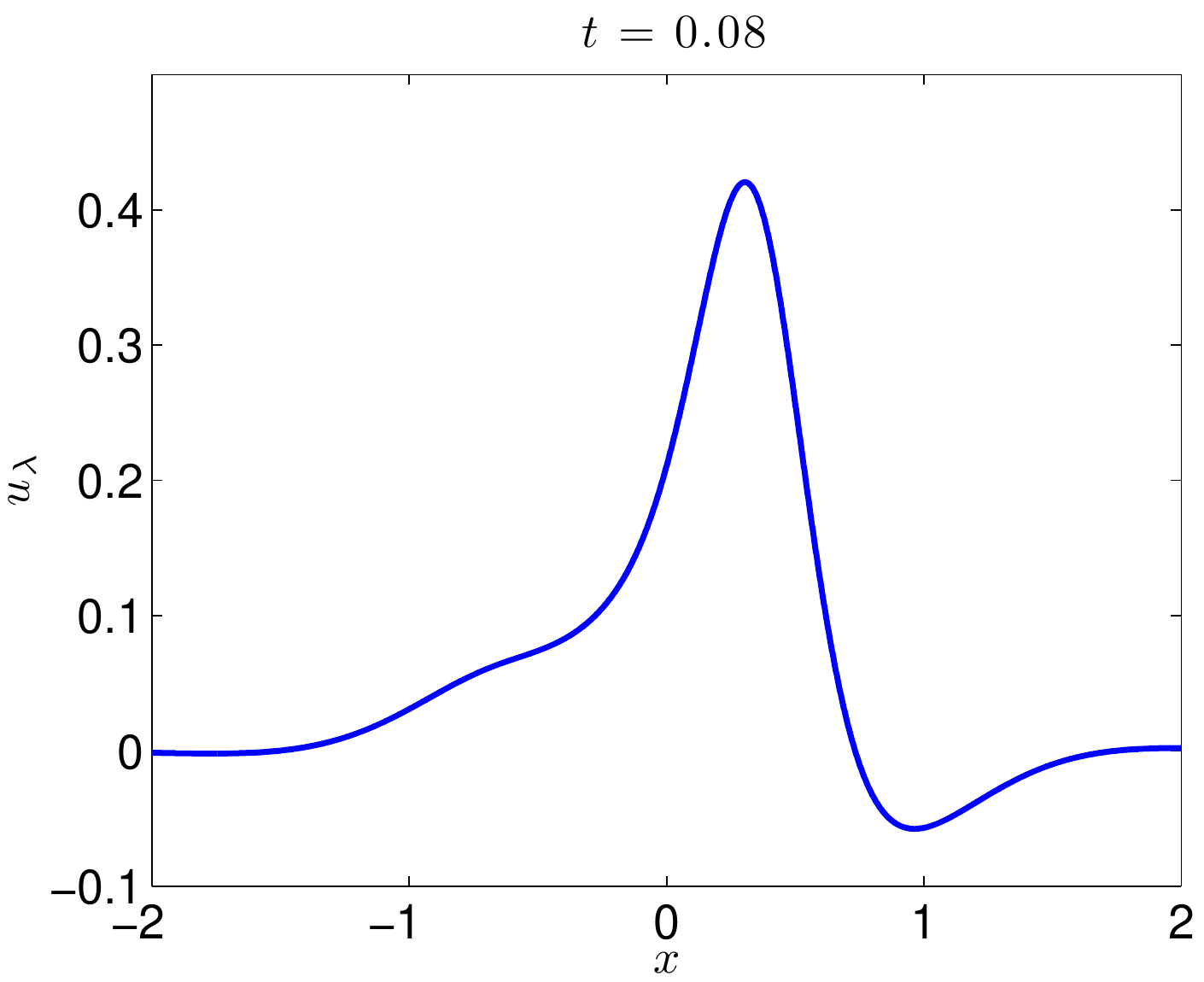}}
\subfigure[]{\includegraphics[width=2.2in]{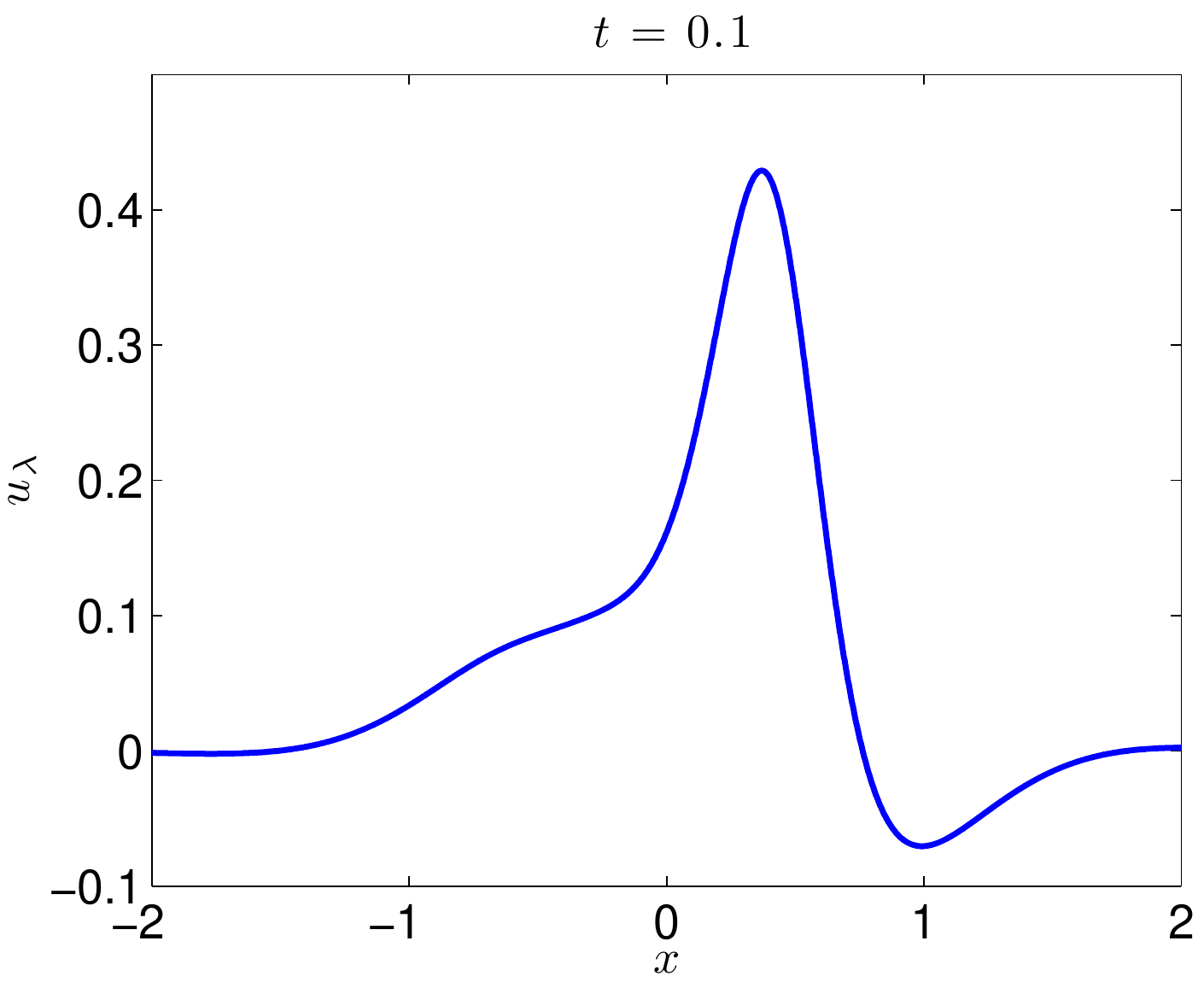}}
\caption{Evolution of the regularized $K(2,2)$
  equation with data given by \eqref{params} with $\lambda = .4$.}
\label{f:evolution_lam04}
\end{figure}

\begin{figure}
\subfigure[]{\includegraphics[width=2.2in]{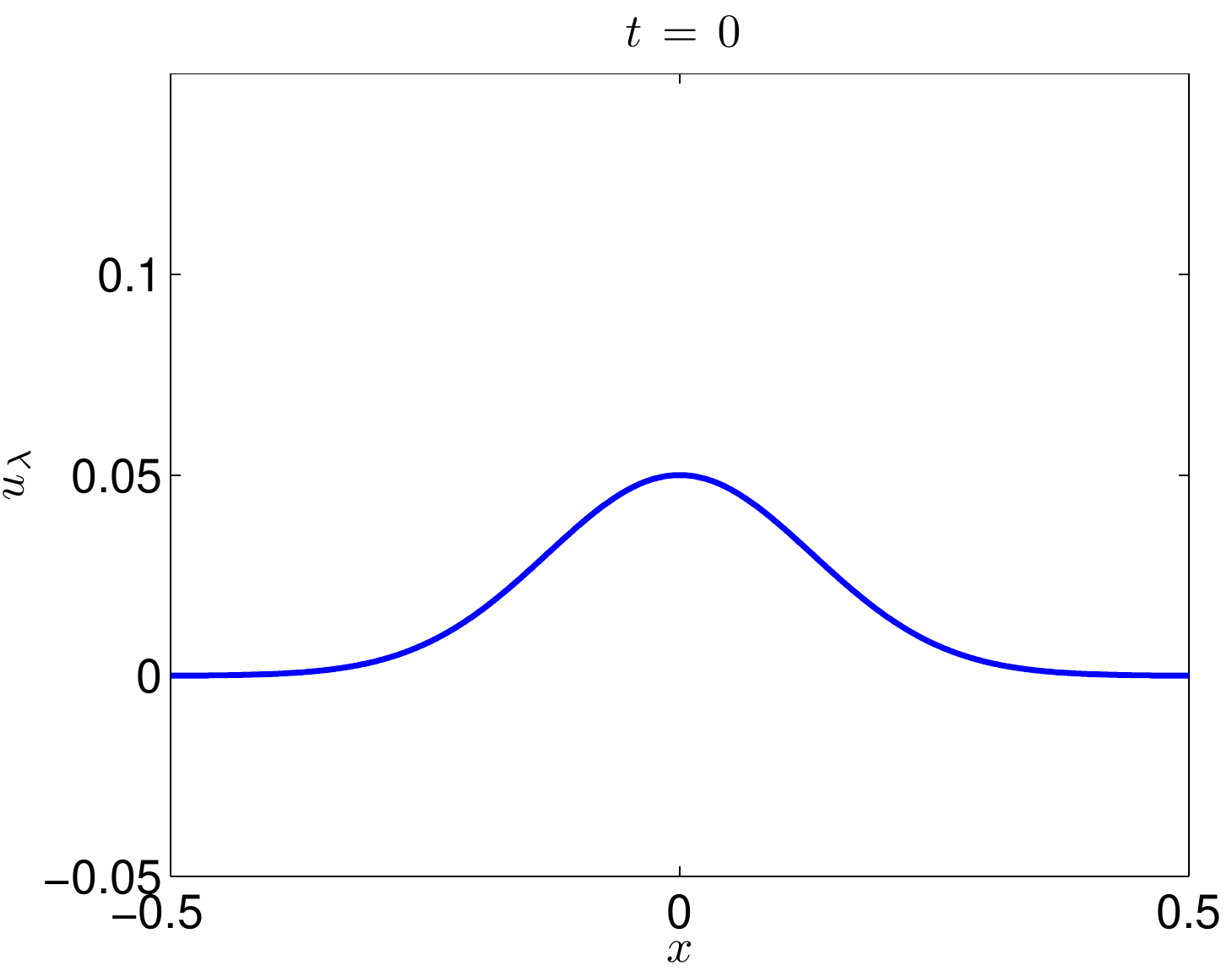}}
\subfigure[]{\includegraphics[width=2.2in]{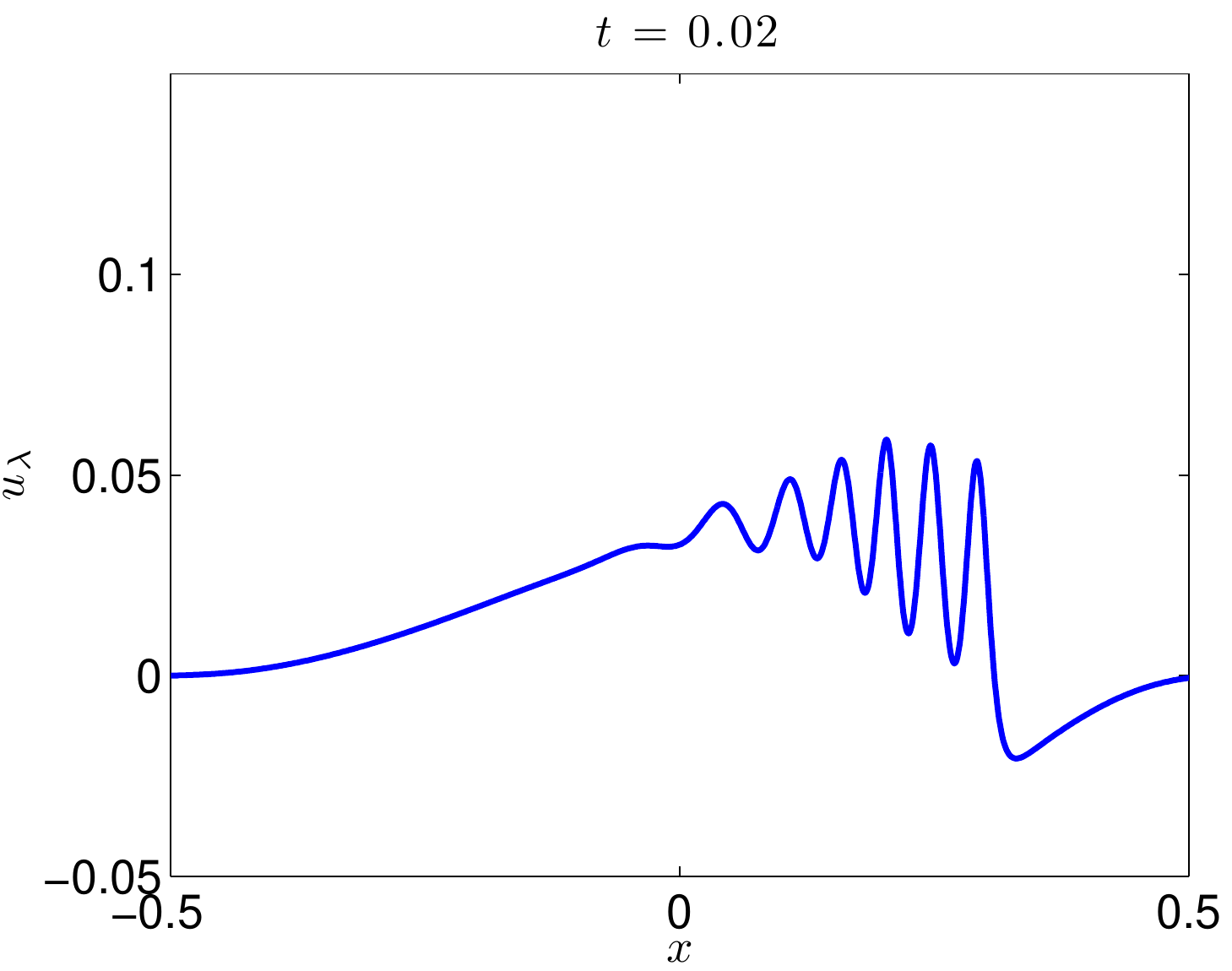}}
\subfigure[]{\includegraphics[width=2.2in]{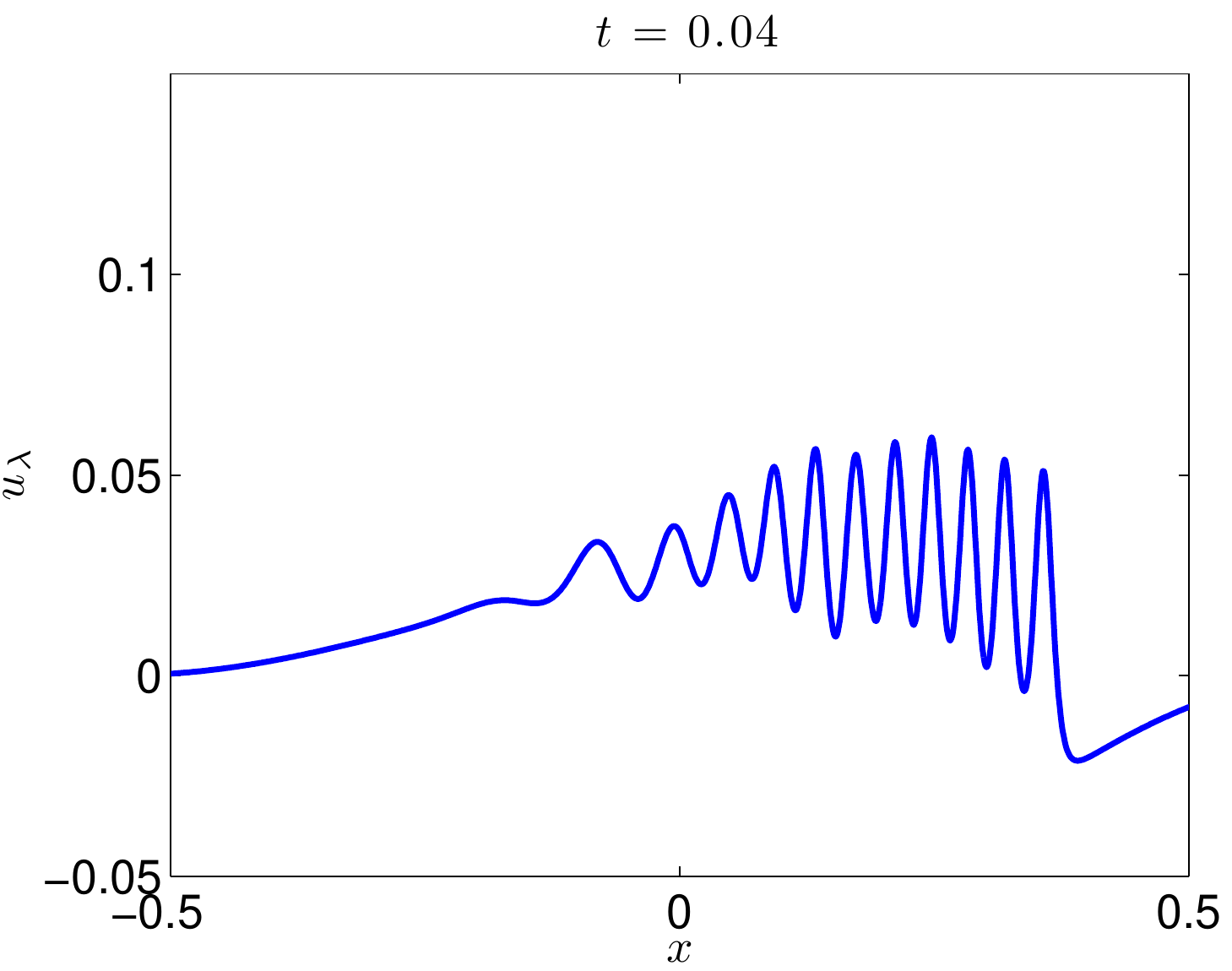}}
\subfigure[]{\includegraphics[width=2.2in]{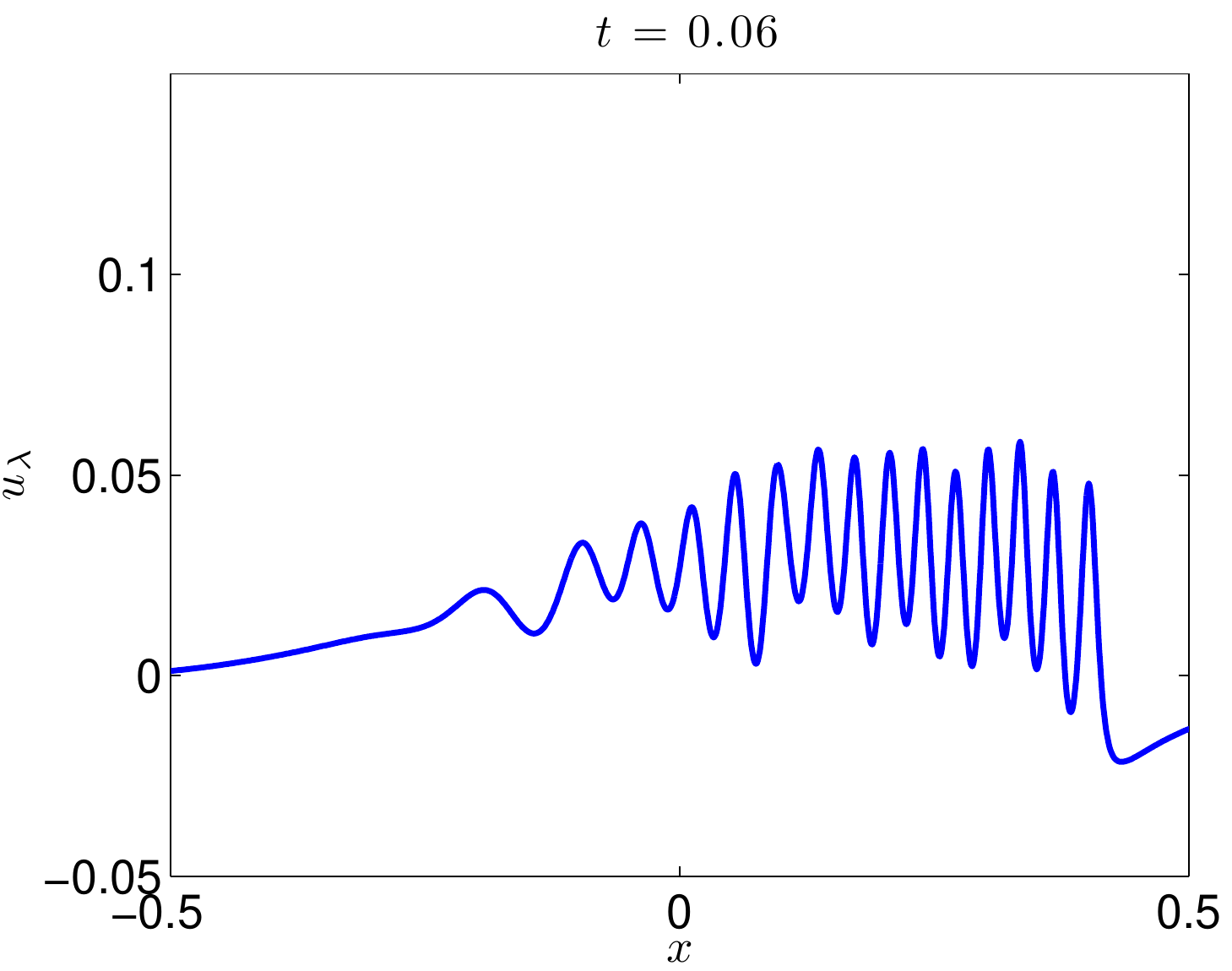}}
\subfigure[]{\includegraphics[width=2.2in]{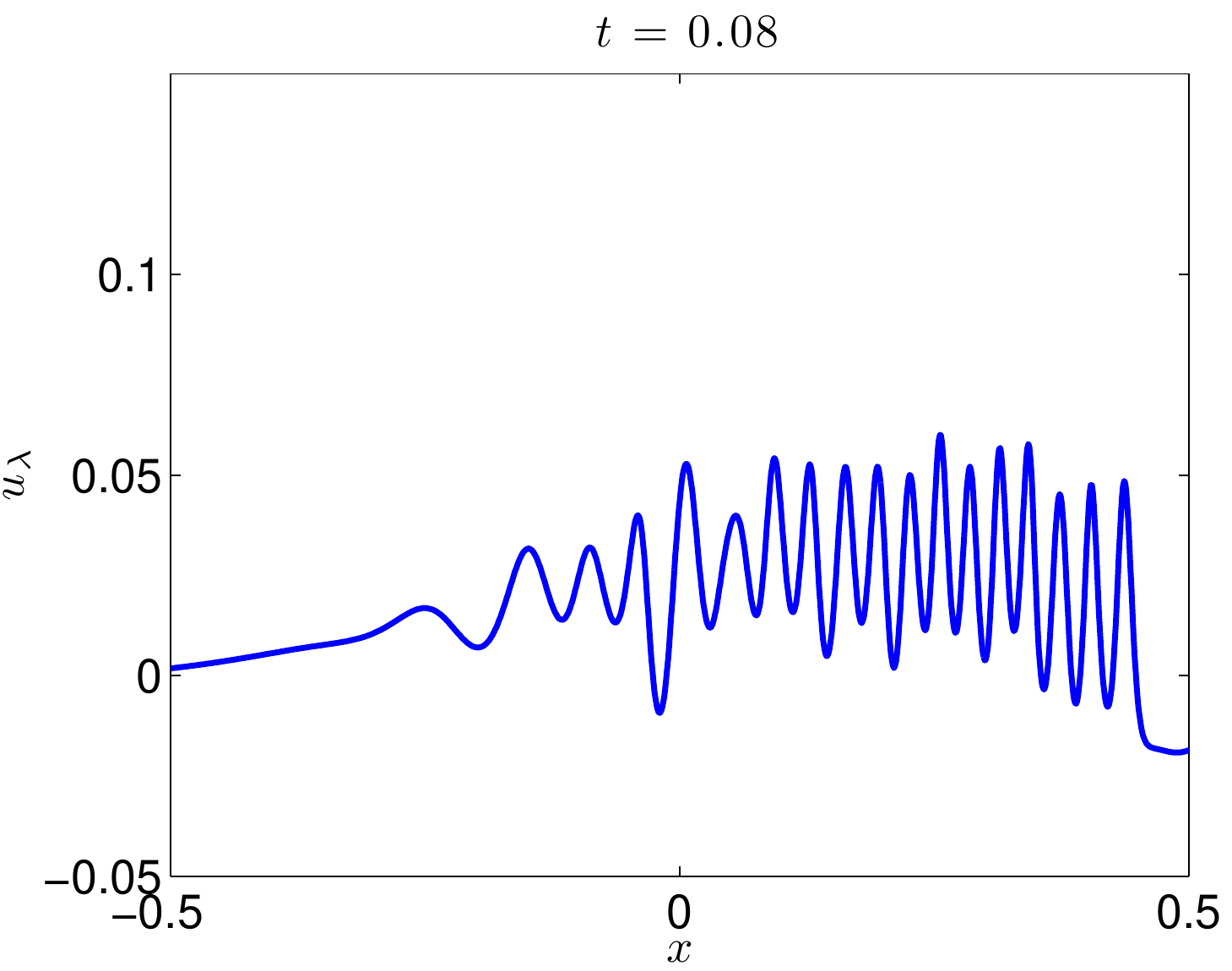}}
\subfigure[]{\includegraphics[width=2.2in]{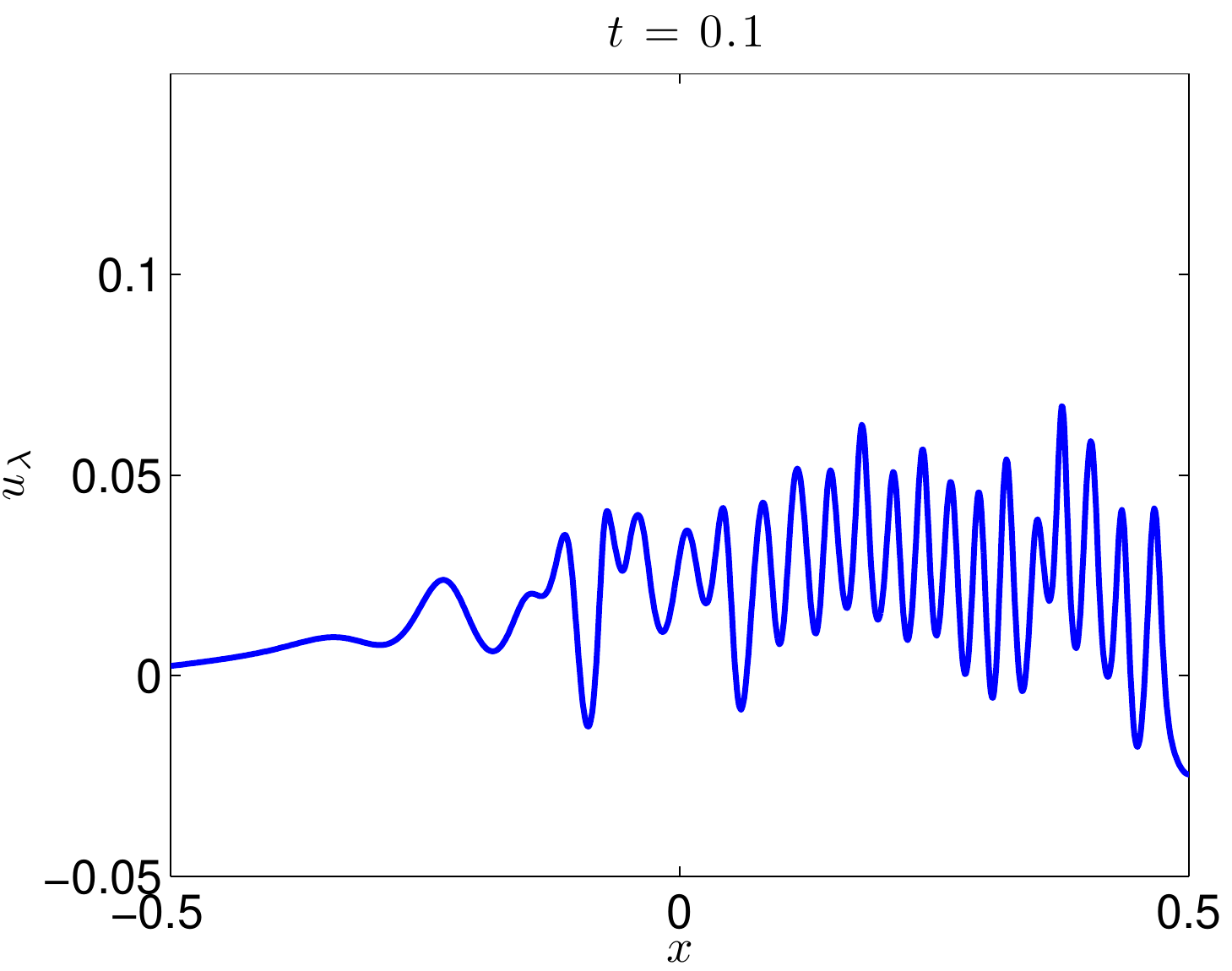}}
\caption{Evolution of the regularized $K(2,2)$
  equation with data given by \eqref{params} with $\lambda = .05$.}
\label{f:evolution_lam005}
\end{figure}

%\subsubsection{Sobolev Norms and Fourier Support}
We present the evolution of the
$L^2$, $\dot{H}^1$, and $\dot{H}^2$ norms in Figure~\ref{f:norms}
of the solutions as functions of time.  Though there is
  little growth in time of the $L^2$ norm, we see orders of magnitude
  jumps in the $\dot{H}^1$ and $\dot{H}^2$ norms at $t = .1$ as we
  send $\lambda \to 0$. 
\begin{figure}
\subfigure[]{\includegraphics[width=2.2in]{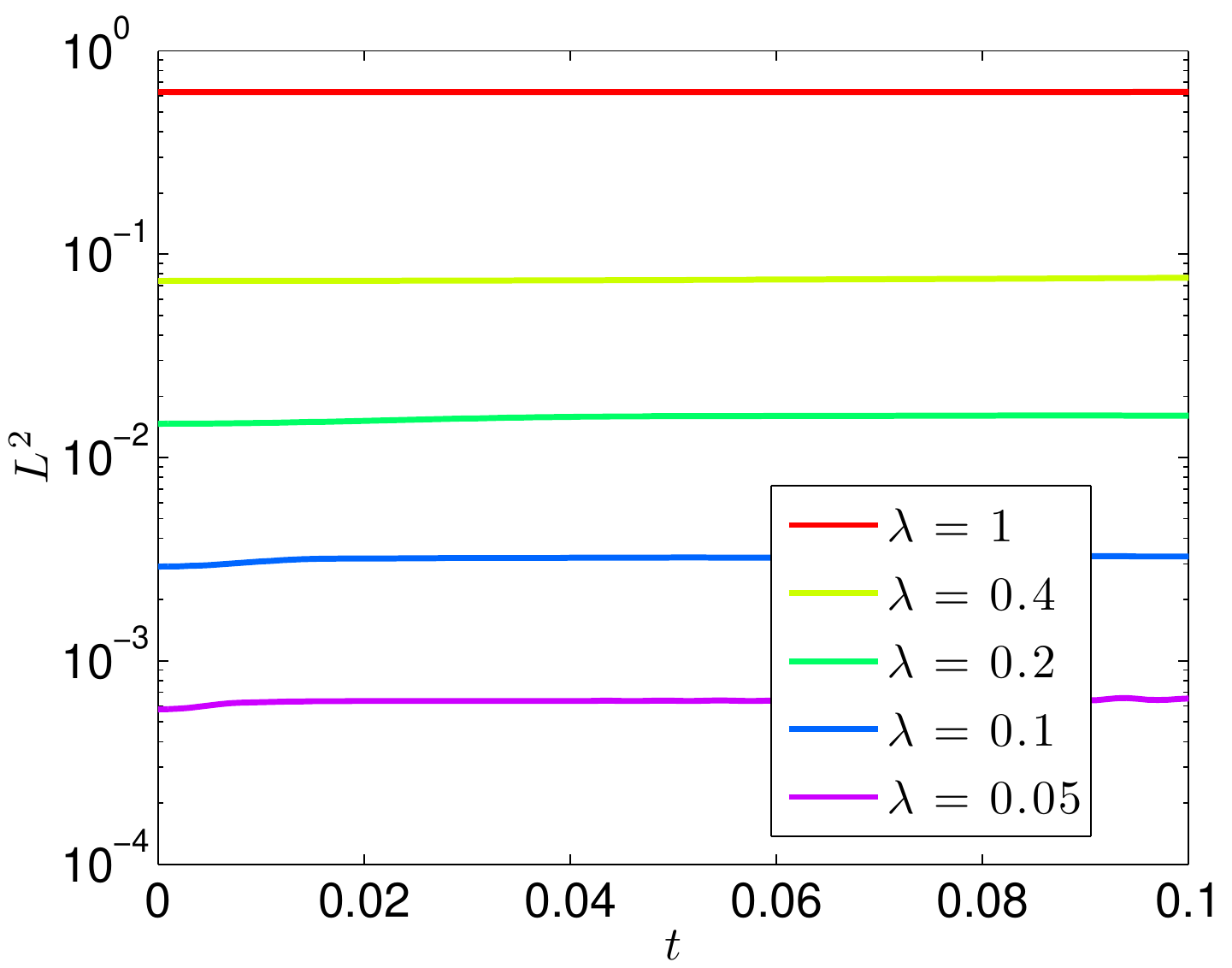}}
\subfigure[]{\includegraphics[width=2.2in]{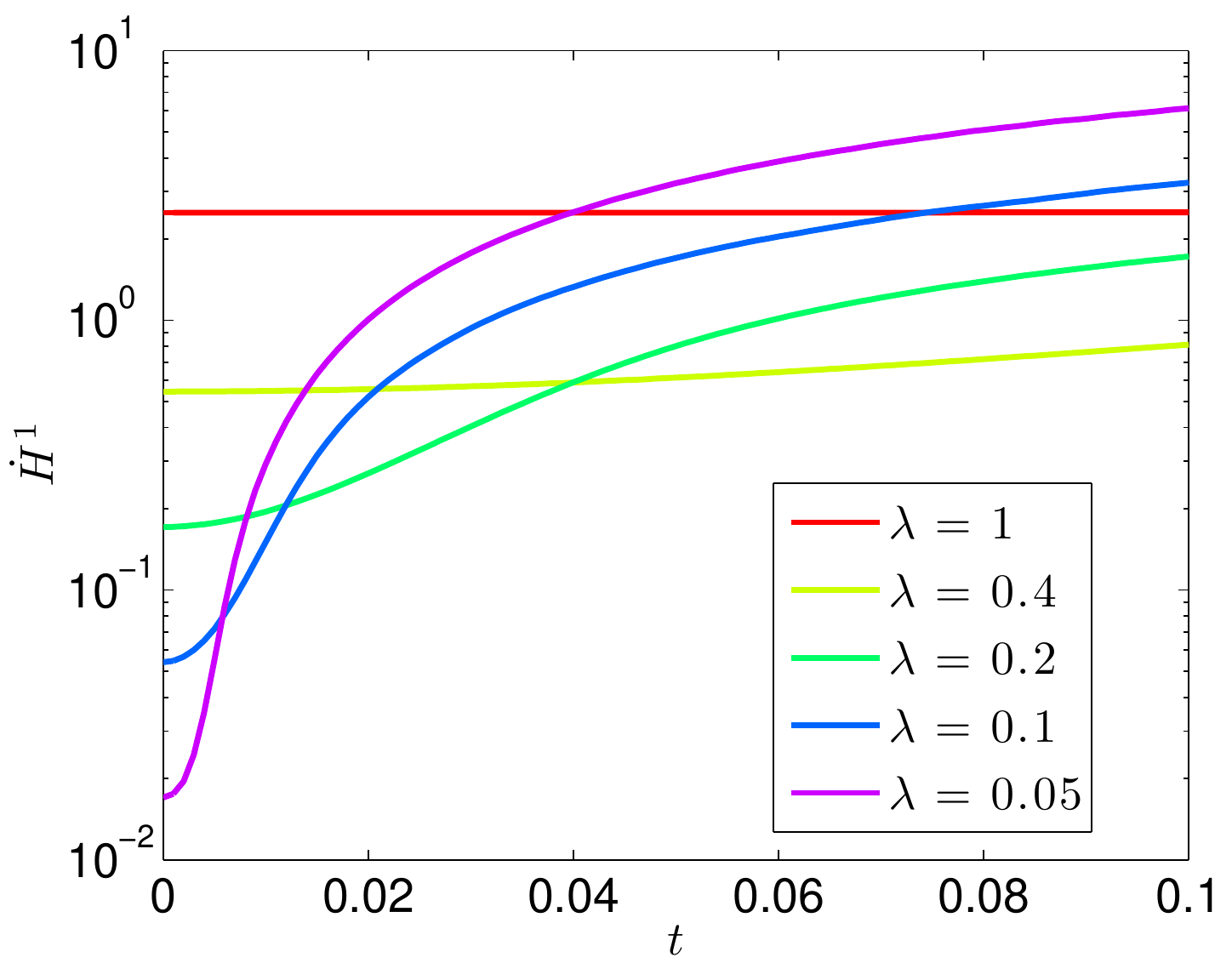}}
\subfigure[]{\includegraphics[width=2.2in]{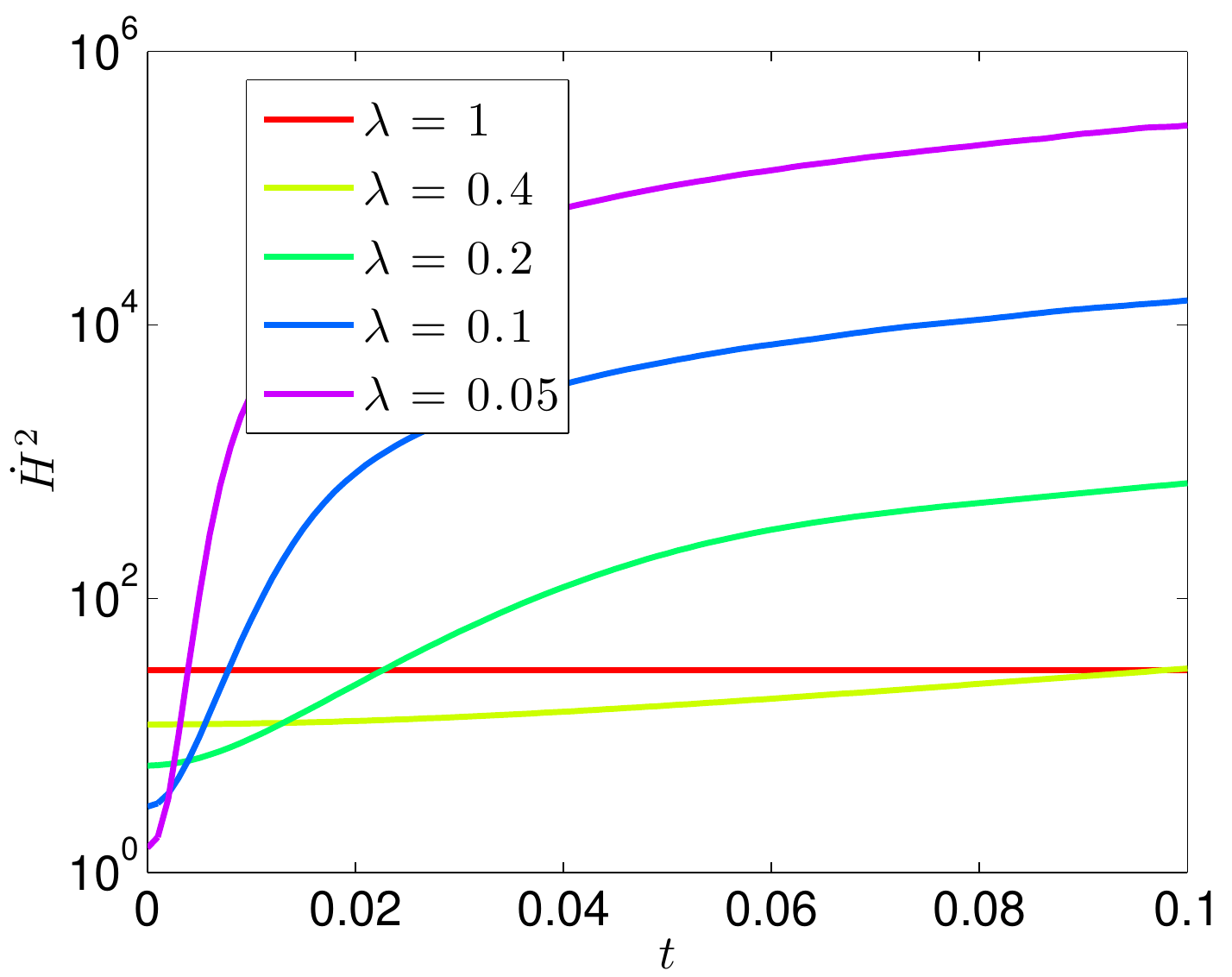}}
\caption{Evolution of several norms for the regularized $K(2,2)$
  equation with data given by \eqref{params}.}
\label{f:norms}
\end{figure}
  We contend that this is evidence of ill-posedness.  We have a
  sequence of initial conditions $f_\lambda$, which have vanishing
$H^2$ norm.  At a fixed $T>0$, the $H^2$ norms are growing as we
  let $\lambda \to 0$.  Thus, there is a loss of continuous dependence
  of the solution map upon the initial data about the $u(x,t) = 0$
  solution.  

  This growth in the $H^2$ norm corresponds to a spreading in Fourier
  space.  Indeed,  Figure~\ref{f:fourier_evolution} shows precisely
  this.  At fixed time, the Fourier support grows as $\delta \to 0$.
  These figures also indicate that except for $\lambda = .04$, all
  simulations are extremely well resolved.  Even the $\lambda = .04$
  simulation is well resolved through $t = .06$, and only slightly
  under resolved at the final time, $t=.1$.

\begin{figure}
\subfigure[]{\includegraphics[width=2.2in]{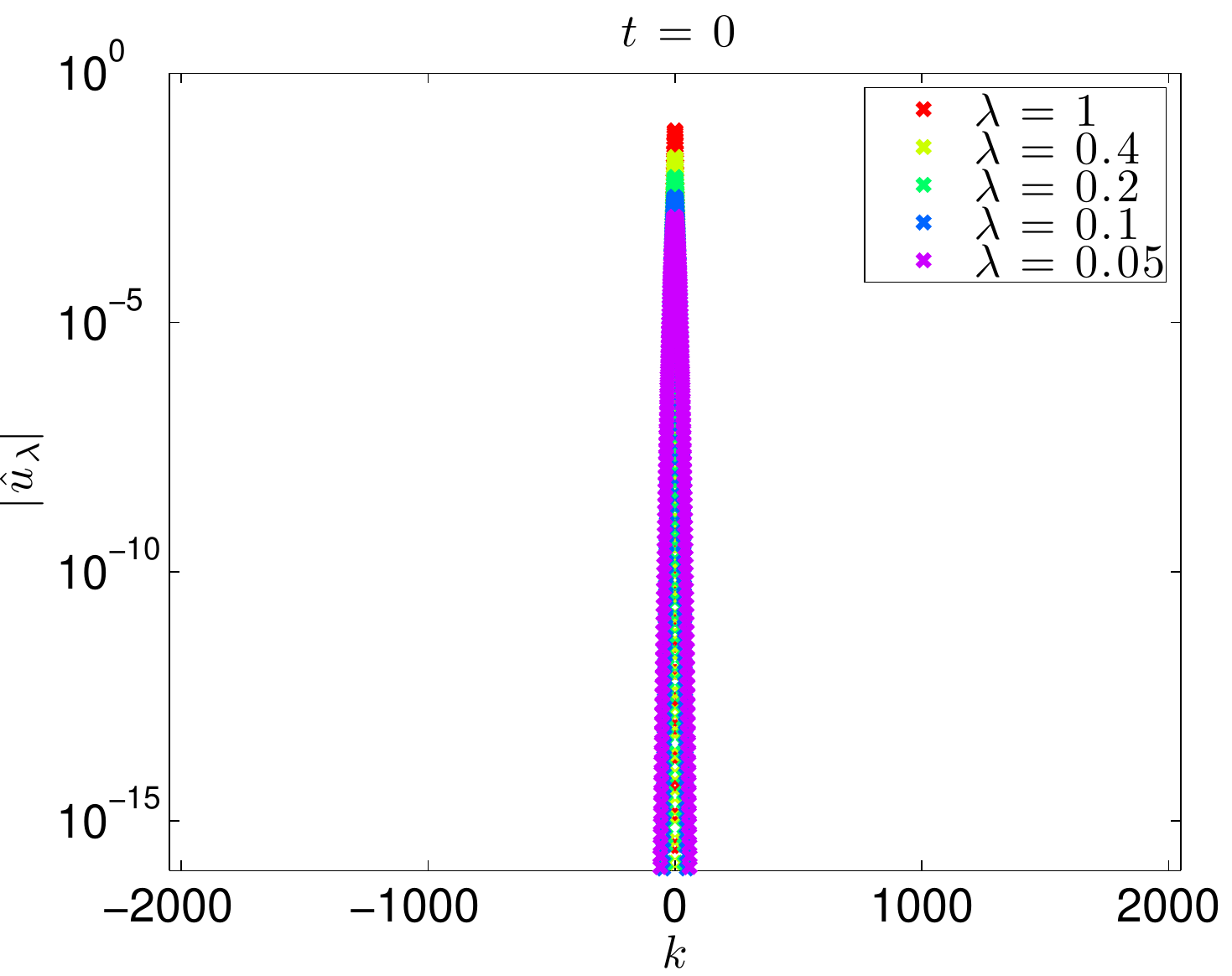}}
\subfigure[]{\includegraphics[width=2.2in]{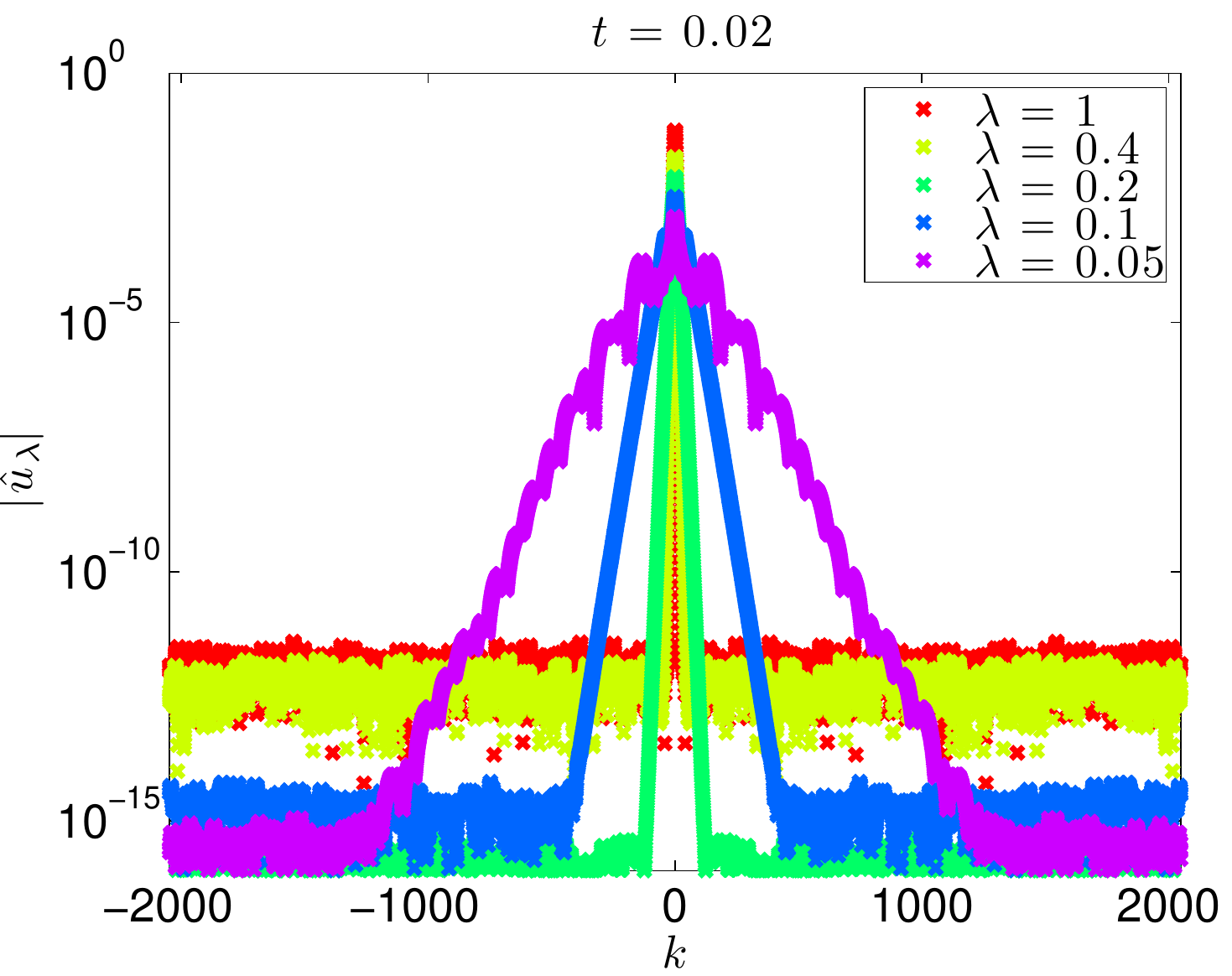}}
\subfigure[]{\includegraphics[width=2.2in]{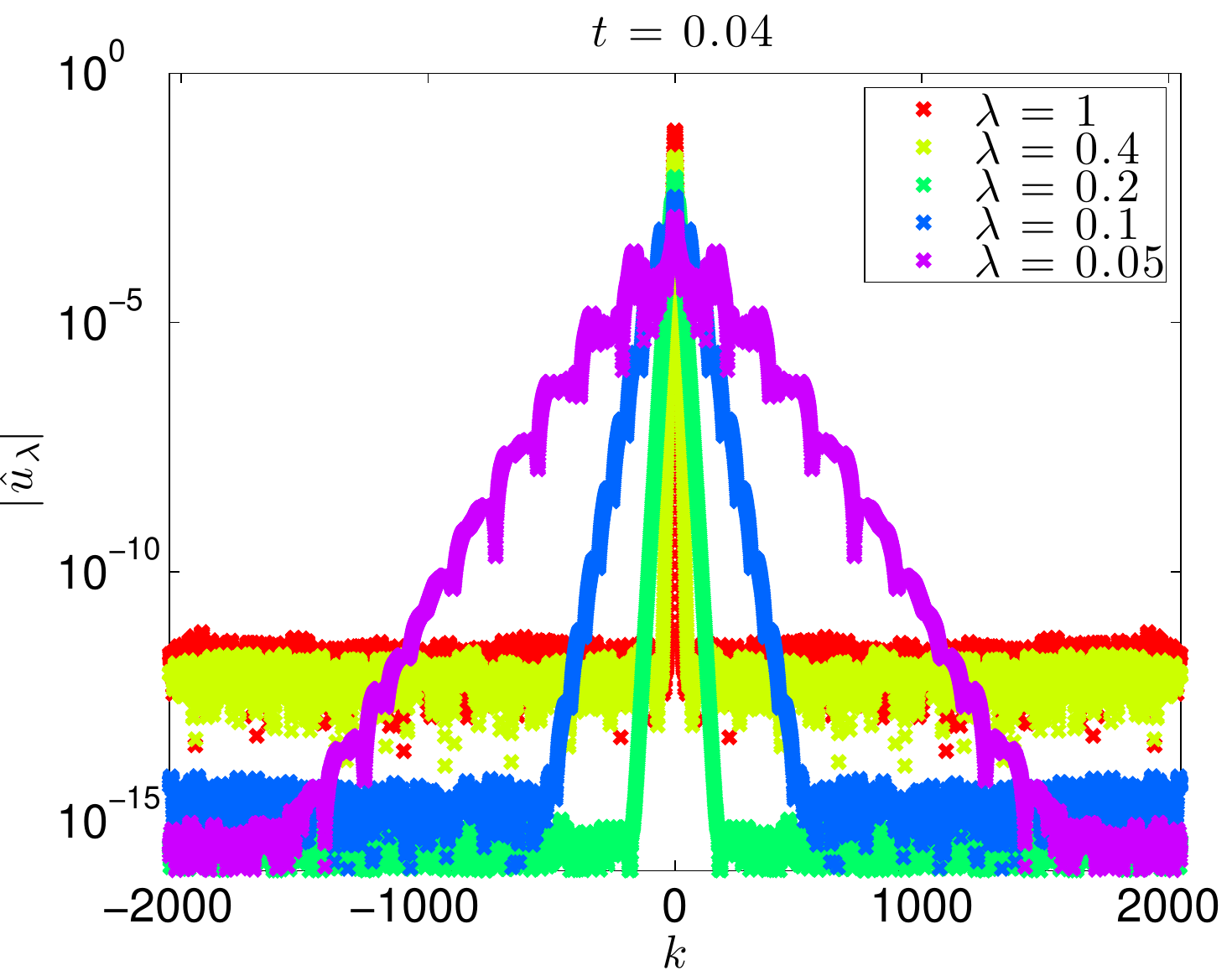}}
\subfigure[]{\includegraphics[width=2.2in]{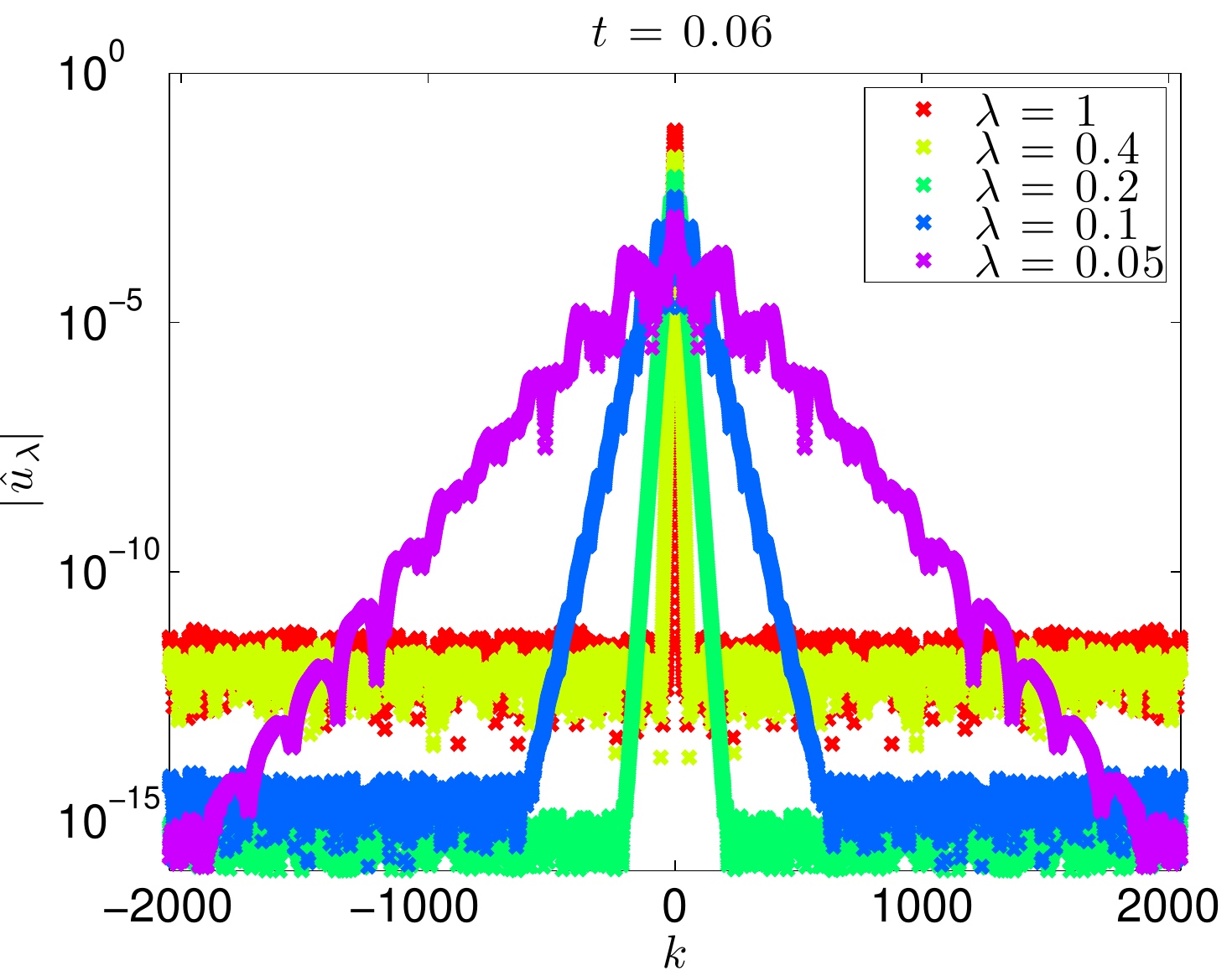}}
\subfigure[]{\includegraphics[width=2.2in]{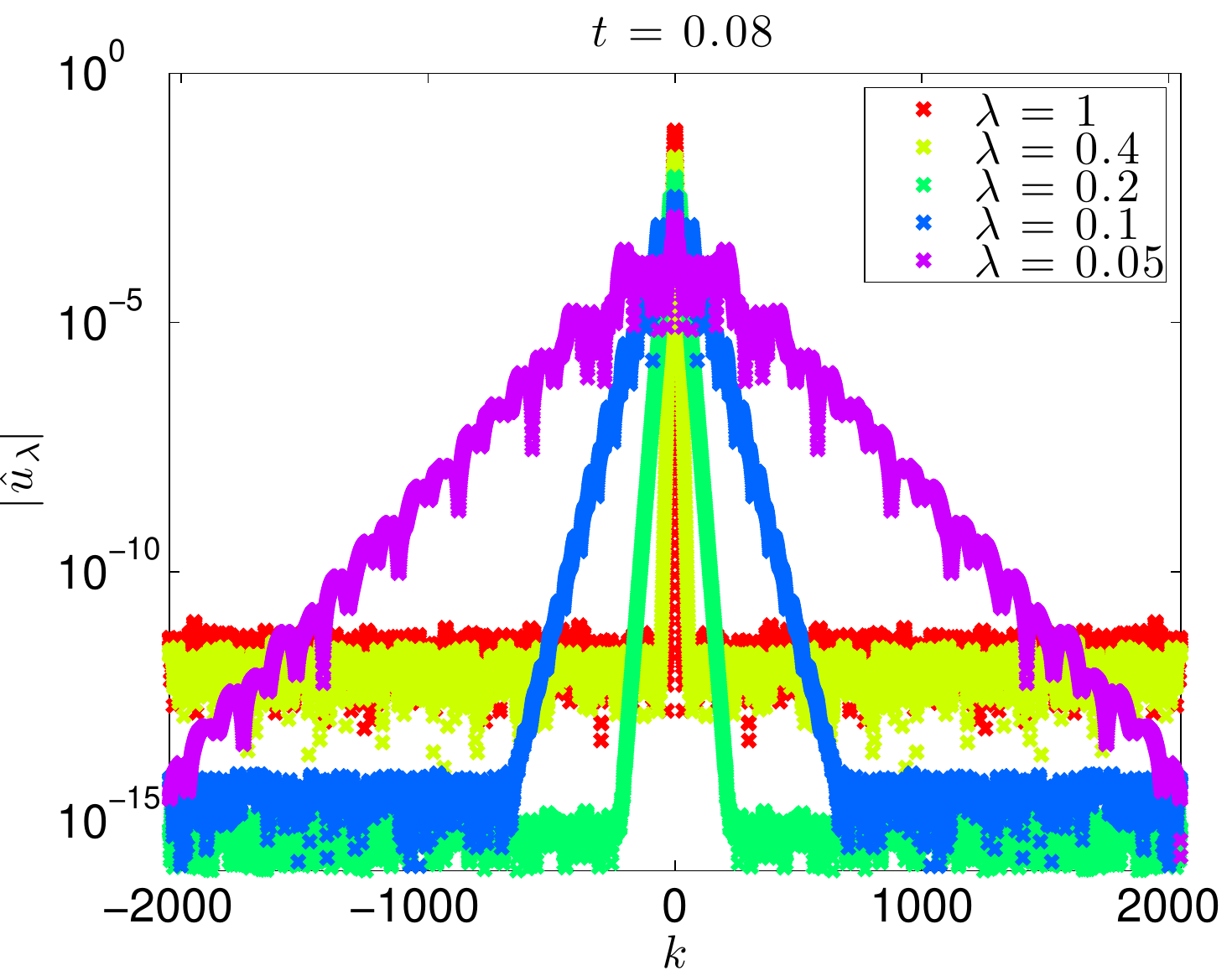}}
\subfigure[]{\includegraphics[width=2.2in]{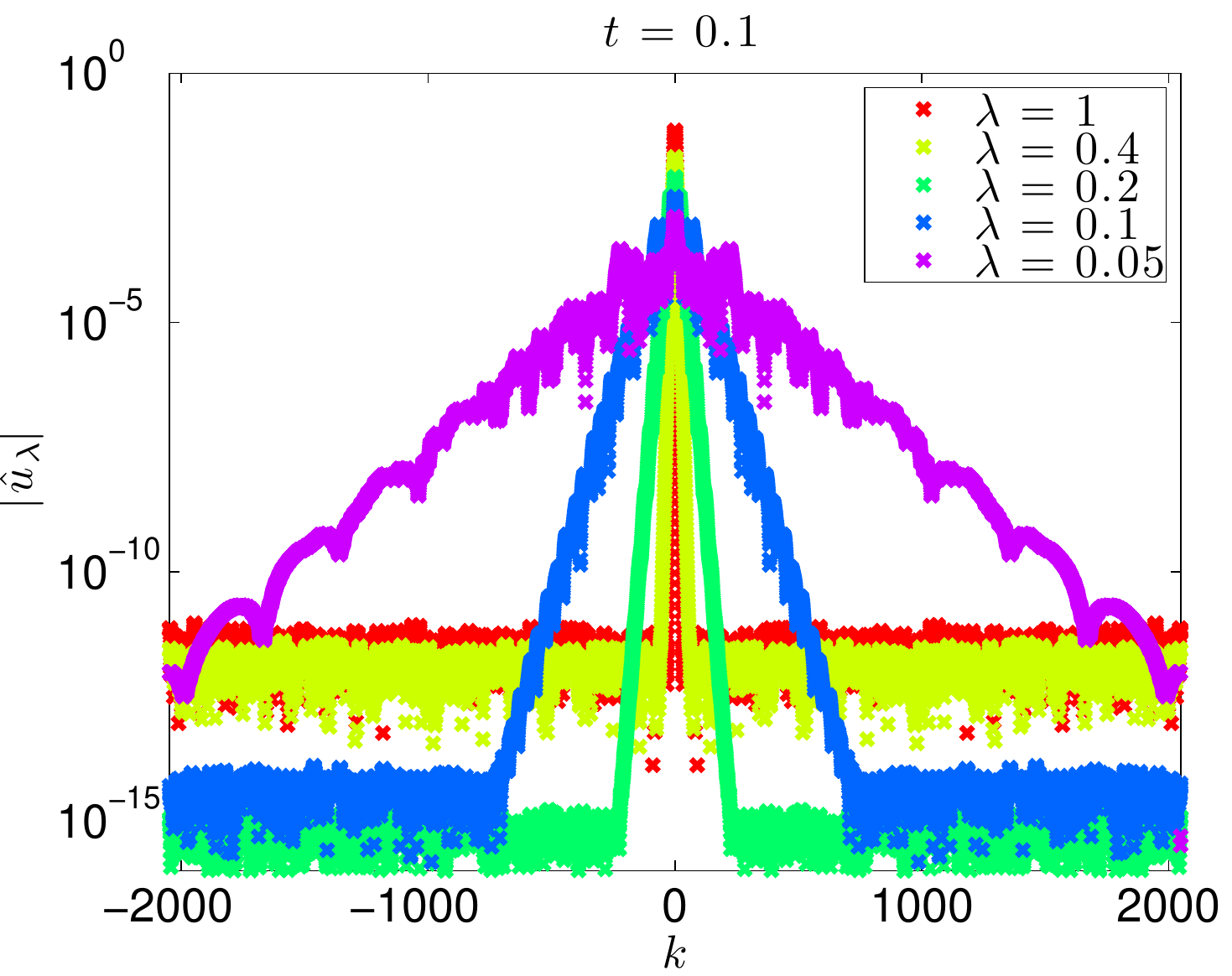}}
\caption{Evolution of the regularized $K(2,2)$
  equation with data given by \eqref{params} with $\lambda = .05$.}
\label{f:fourier_evolution}
\end{figure}

\subsection{Details of numerical methods}
\label{s:numerics}

As noted, the simulations presented in Section \ref{SC} were generated
using a fully de-aliased pseudospectral discretization with
Crank-Nicholson time stepping.  This was performed in {\sc Matlab}
and {\tt fsolve} was used to solve the nonlinear system at each time
step with a tolerance of $10^{-8}$.  8192 modes were resolved in these
simulations, with a time step of $\Delta t = .001$.  

The invariant $\ds \int u(x,t) dx$ was conserved to at least eight
significant figures in all the computed cases, as shown in Table \ref{t:mass_conservation}.

\begin{table}
\centering
\caption{Numerical approximation of $\int u(x,t) dx$ at the outputted
  times for different values of $\lambda$}
\label{t:mass_conservation}
  \begin{tabular}{ c c | c  c  c  c  c  c }
\hline
\hline
   & & \multicolumn{5}{|c}{$\lambda$} \\
   & & 1 & 0.4 & 0.2 & 0.1 & 0.05 \\	  
    \hline 
    \multirow{6}{*}{$t$} & 0.00 & 0.886226925453 & 0.261191032665 & 0.103653730004 & 0.0411350100123 & 0.0163244395416 \\
    &0.02  & 0.886226925477 & 0.261191032667 & 0.103653730004 & 0.0411350100123& 0.0163244395416  \\ 
    & 0.04  & 0.886226925489 & 0.261191032674 & 0.103653730004 & 0.0411350100123 & 0.0163244395416 \\
&0.06  & 0.886226925478 & 0.261191032676 & 0.103653730004& 0.0411350100123 & 0.0163244395416 \\
&0.08  & 0.886226925497 & 0.261191032673 & 0.103653730004 & 0.0411350100123 & 0.0163244395416 \\	
&0.10  & 0.886226925509 & 0.261191032673 & 0.103653730004 &
0.0411350100123 & 0.0163244395416\\
\hline
\end{tabular}
\end{table}

\subsection{Convergence of Norms}

Since we are concerned with the value of $\dot H^2$ at the end of our
simulation, it is also important to confirm that this is converging to
a fixed value as we refine the spatial and temporal resolution of our
simulations.   For the initial condition with $\lambda = 0.2$,  Table
\ref{t:h2_convergence_lam2} gives the values of the
$\dot H^2$ norm for various choices of the number of resolved Fourier modes,
$N$, and the time step, $\Delta t$.  With only 512 resolved modes, we
appear to be fully resolved in space, and there is little gain in
accuracy when $N$ is increased.  Of course, for other initial
conditions, more than 512 modes may be needed to resolve the
simulation.  

\begin{table}
\centering
\caption{Values of $\dot H^2$ at $t=.1$ corresponding to the $\lambda
  = 0.2$ simulation with different numbers of resolved Fourier modes, $N$,
  and time steps, $\Delta t$.}
\label{t:h2_convergence_lam2}
\begin{tabular}{c r | l l l }
\hline
\hline
& & \multicolumn{3}{|c}{$\Delta t$} \\
 &     &  0.004 & 0.002 & 0.001 \\ 
\hline
    \multirow{5}{*}{$N$}&   512 &  675.27623578 & 692.22147559 & 696.5301160 \\ 
& 1024 &  675.27623585 & 692.22147576 & 696.5301162 \\ 
& 2048 &  675.27623587 & 692.22147576 & 696.5301162 \\ 
& 4096 &  675.27623586 & 692.22147576 & 696.5301162 \\ 
& 8192 &  675.27623586 & 692.22147576 & 696.5301162 \\ 
\hline
\end{tabular}
\end{table}

For the three values of $\Delta t$, we appear to have achieved at
least one significant digit of accuracy and as we reduce the time
step, the variations of $\dot H^2$ become smaller.   Since we are using a
Crank-Nicholson scheme, we expect ${O}(\Delta t^2)$ convergence.
If we perform Richardson extrapolation, the estimated value of $\dot H^2$ based on the  $\Delta t = .004$ and
$\Delta t =.002$ simulations is 697.87.  Comparing the $\Delta t =
.002$ and $\Delta t = .001$ simulations, the estimated value is
697.97.  Separate computations, performed with just 512 grid points
for expediency, show that with $\Delta t = 0.0005$ the norm takes the
value 697.96, and for $\Delta t = 0.00025$, the value is 697.99.

An examination of the $\lambda = 0.1$ case is similar.  For the same
three values of $\Delta t$, the $\dot H^2$ norms, given in Table
\ref{t:h2_convergence_lam1}, are in agreement on the order of
magnitude.  Richardson extrapolation using the data at $\Delta t=.004$ and
$\Delta t = .002$ predicts a value of 15148, while the predicted value
based on the $\Delta t = .002$ and $\Delta t = .001$ value is 15278.
Independent simulations with 2048 grid points yield a value of 15253
when $\Delta t = 0.0005$ and a value of 15302 when $\Delta t =
0.00025$.  Thus, we believe that our choice of discretization
parameters, $N=8192$ and $\Delta t = 0.001$, for the results presented
in Section 4.2 yield meaningful
measurements of $\dot H^2$ that are accurate to at least one
signfigant figure.  This is sufficient for our study of ill-posedness.  

\begin{table}
\centering
\caption{Values of $\dot H^2$ at $t=.1$ corresponding to  the $\lambda = 0.1$
  simulation with different numbers of resolved Fourier modes, $N$,
  and time steps, $\Delta t$.}
\label{t:h2_convergence_lam1}
\begin{tabular}{c r | l l l }
\hline
\hline
& & \multicolumn{3}{|c}{$\Delta t$} \\
 &     &  0.004 & 0.002 & 0.001 \\ 
\hline
    \multirow{3}{*}{$N$} & 2048 &  12452.791732 & 14474.328299 & 15077.20691 \\ 
& 4096 &  12452.791732 & 14474.328298 & 15077.20691 \\ 
& 8192 &  12452.791732 & 14474.328299 & 15077.20691 \\
\hline
\end{tabular}
\end{table}

\newpage
\phantom{hi!}
{
\nocite{*}
\bibliographystyle{abbrv}
\bibliography{self-similar.bib}
}

\end{document}

%% file: gsmacros.tex
\newcommand{\norm}[1]{\lVert#1\rVert}
\newcommand{\paren}[1]{\left(#1\right)}

\newcommand{\exist}{\mathrm{exist}}

%% file: currentversion.bbl
\def\cprime{$'$} \def\cprime{$'$}
\begin{thebibliography}{10}

\bibitem{Ambrose-Masmoudi}
D.~M. Ambrose and N.~Masmoudi.
\newblock Well-posedness of 3{D} vortex sheets with surface tension.
\newblock {\em Commun. Math. Sci.}, 5(2):391--430, 2007.

\bibitem{Ambrose-Wright2}
D.~M. Ambrose and J.~D. Wright.
\newblock Traveling waves and weak solutions for an equation with degenerate
  dispersion.
\newblock Preprint.

\bibitem{Ambrose-Wright1}
D.~M. Ambrose and J.~D. Wright.
\newblock Preservation of support and positivity for solutions of degenerate
  evolution equations.
\newblock {\em Nonlinearity}, 23(3):607--620, 2010.

\bibitem{Barenblatt}
G.~I. Barenblatt.
\newblock On self-similar motions of a compressible fluid in a porous medium.
\newblock {\em Akad. Nauk SSSR. Prikl. Mat. Meh.}, 16:679--698, 1952.

\bibitem{Camassa-Holm}
R.~Camassa and D.~D. Holm.
\newblock An integrable shallow water equation with peaked solitons.
\newblock {\em Phys. Rev. Lett.}, 71(11):1661--1664, 1993.

\bibitem{Chertock-Levy}
A.~Chertock and D.~Levy.
\newblock Particle methods for dispersive equations.
\newblock {\em J. Comput. Phys.}, 171(2):708--730, 2001.

\bibitem{Craig-Goodman}
W.~Craig and J.~Goodman.
\newblock Linear dispersive equations of {A}iry type.
\newblock {\em J. Differential Equations}, 87(1):38--61, 1990.

\bibitem{Craig-Kappeler-Strauss}
W.~Craig, T.~Kappeler, and W.~Strauss.
\newblock Gain of regularity for equations of {K}d{V} type.
\newblock {\em Ann. Inst. H. Poincar\'e Anal. Non Lin\'eaire}, 9(2):147--186,
  1992.

\bibitem{Dai-Huo}
H.-H. Dai and Y.~Huo.
\newblock Solitary shock waves and other travelling waves in a general
  compressible hyperelastic rod.
\newblock {\em R. Soc. Lond. Proc. Ser. A Math. Phys. Eng. Sci.},
  456(1994):331--363, 2000.

\bibitem{Defrutos}
J.~de~Frutos, M.~{\'A}. L{\'o}pez~Marcos, and J.~M. Sanz-Serna.
\newblock A finite difference scheme for the {$K(2,2)$} compacton equation.
\newblock {\em J. Comput. Phys.}, 120(2):248--252, 1995.

\bibitem{Daraio1}
F.~Fraternali, M.~A. Porter, and C.~Daraio.
\newblock {Optimal Design of Composite Granular Protectors}.
\newblock {\em {Mech. Adv. Mater. Struct.}}, {17}({1}):{1--19}, {2010}.

\bibitem{Galaktionov2}
V.~A. Galaktionov.
\newblock Nonlinear dispersion equations: smooth deformations, compactions, and
  extensions to higher orders.
\newblock {\em Zh. Vychisl. Mat. Mat. Fiz.}, 48(10):1859, 2008.

\bibitem{Galaktionov1}
V.~A. Galaktionov and S.~I. Pokhozhaev.
\newblock Third-order nonlinear dispersion equations: shock waves, rarefaction
  waves, and blow-up waves.
\newblock {\em Zh. Vychisl. Mat. Mat. Fiz.}, 48(10):1819--1846, 2008.

\bibitem{Goodman-Lax}
J.~Goodman and P.~D. Lax.
\newblock On dispersive difference schemes. {I}.
\newblock {\em Comm. Pure Appl. Math.}, 41(5):591--613, 1988.

\bibitem{Kenig}
C.~E. Kenig, G.~Ponce, and L.~Vega.
\newblock The initial value problem for the general quasi-linear
  {S}chr\"odinger equation.
\newblock In {\em Recent developments in nonlinear partial differential
  equations}, volume 439 of {\em Contemp. Math.}, pages 101--115. Amer. Math.
  Soc., Providence, RI, 2007.

\bibitem{Ketcheson}
D.~Ketcheson.
\newblock {\em High Order Strong Stability Preserving Time Integrators and
  Numerical Wave Propagation Methods for Hyperbolic PDEs}.
\newblock 2009.

\bibitem{Levy-Shu-Yan}
D.~Levy, C.-W. Shu, and J.~Yan.
\newblock Local discontinuous {G}alerkin methods for nonlinear dispersive
  equations.
\newblock {\em J. Comput. Phys.}, 196(2):751--772, 2004.

\bibitem{Li-Olver-Rosenau}
Y.~A. Li, P.~J. Olver, and P.~Rosenau.
\newblock Non-analytic solutions of nonlinear wave models.
\newblock In {\em Nonlinear theory of generalized functions ({V}ienna, 1997)},
  volume 401 of {\em Chapman \& Hall/CRC Res. Notes Math.}, pages 129--145.
  Chapman \& Hall/CRC, Boca Raton, FL, 1999.

\bibitem{Nesterenko}
V.~F. Nesterenko.
\newblock {\em Dynamics of hetereogeneous materials}.
\newblock Springer, New York, NY, 2001.

\bibitem{Daraio2}
L.~Ponson, N.~Boechler, Y.~M. Lai, M.~A. Porter, P.~G. Kevrekidis, and
  C.~Daraio.
\newblock {Nonlinear waves in disordered diatomic granular chains}.
\newblock {\em {Phys. Rev. E}}, {82}({2, Part 1}), {AUG 12} {2010}.

\bibitem{Kevrekidis}
M.~Porter, C.~Daraio, I.~Szelengowicz, E.~Herbold, and P.~Kevrekidis.
\newblock Highly nonlinear solitary waves in heterogeneous periodic granular
  media.
\newblock {\em Physica D}, 2009.

\bibitem{Rosenau-PRL1994}
P.~Rosenau.
\newblock Nonlinear dispersion and compact structures.
\newblock {\em Phys. Rev. Lett.}, 73(13):1737--1741, 1994.

\bibitem{Rosenau-PLA00}
P.~Rosenau.
\newblock Compact and noncompact dispersive patterns.
\newblock {\em Phys. Lett. A}, 275(3):193--203, 2000.

\bibitem{Rosenau-Hyman}
P.~Rosenau and J.~Hyman.
\newblock Compactons: solitons with finite wavelength.
\newblock {\em Phys. Rev. Lett.}, (70):564--567, 1993.

\bibitem{Rosenau-Kashdan}
P.~Rosenau and E.~Kashdan.
\newblock Emergence of compact structures in a klein-gordon model.
\newblock {\em Phys. Rev. Lett.}, (104), 2010.

\bibitem{Rubinstein}
J.~Rubinstein and J.~B. Keller.
\newblock Sedimentation of a dilute suspension.
\newblock {\em Phys. Fluids A}, 1(4):637--643, 1989.

\bibitem{Rus2}
F.~Rus and F.~R. Villatoro.
\newblock Pad\'e numerical method for the {R}osenau-{H}yman compacton equation.
\newblock {\em Math. Comput. Simulation}, 76(1-3):188--192, 2007.

\bibitem{Rus1}
F.~Rus and F.~R. Villatoro.
\newblock Self-similar radiation from numerical {R}osenau-{H}yman compactons.
\newblock {\em J. Comput. Phys.}, 227(1):440--454, 2007.

\bibitem{Simpson}
G.~Simpson, M.~Spiegelman, and M.~I. Weinstein.
\newblock Degenerate dispersive equations arising in the study of magma
  dynamics.
\newblock {\em Nonlinearity}, 20(1):21--49, 2007.

\bibitem{Wayne-Wright}
C.~E. Wayne and J.~D. Wright.
\newblock Higher order modulation equations for a {B}oussinesq equation.
\newblock {\em SIAM J. Appl. Dyn. Syst.}, 1(2):271--302 (electronic), 2002.

\bibitem{Wright}
J.~D. Wright.
\newblock {\em Higher order corrections to the {K}d{V} approximation for water
  waves}.
\newblock ProQuest LLC, Ann Arbor, MI, 2004.
\newblock Thesis (Ph.D.)--Boston University.

\end{thebibliography}
